\documentclass[reqno,10pt]{amsart}

\usepackage[T1]{fontenc}
\usepackage[english]{babel}
\usepackage{amsmath, amsthm, amssymb}
\usepackage{mathrsfs, esint, bbm}   
\usepackage{fancyhdr}               
\usepackage{fixltx2e}               
\usepackage{enumitem}               
\usepackage[all,cmtip]{xy}          
\usepackage{graphicx}
\usepackage{bm} 
\usepackage[normalem]{ulem} 

\usepackage{color}
\def\R {\mathbb{R}}

\newtheorem{theorem}{Theorem}
\newtheorem{definition}{Definition}[section]
\newtheorem{lemma}{Lemma}[section]
\newtheorem{proposition}{Proposition}[section]
\newtheorem{remark}{Remark}[section]

\newtheorem{corollary}{Corollary}[section]
\numberwithin{equation}{section}

\DeclareMathOperator{\supp}{supp}

\newcommand{\abs}[1]{\left|#1\right|}
\newcommand{\norm}[1]{\left\|#1\right\|}

\newcommand{\eps}{\varepsilon}

\newcommand{\dd}{\,\mathrm{d}}
\renewcommand{\d}{\mathrm{d}}
\newcommand{\Ds}{\left(-\Delta\right)^s }

\newcommand{\Dusm}{\left(-\Delta\right)^{\frac{1+s}{2}} }
\makeindex

\renewcommand{\div}{\operatorname{div}}

\begin{document}

\title[$\alpha$-Entropy solutions for Nonlocal Thin films]{On $\alpha$-Entropy Solutions of a Nonlocal Thin Film Equation: Existence and finite speed of propagation}

\pagestyle{myheadings}
\author{Antonio Segatti}
\address{Dipartimento di Matematica ``F. Casorati'', via Ferrata 1, I--27100 Pavia, Italy}
\email{antonio.segatti@unipv.it}
\urladdr{}

\author{Roman Taranets}
\address{Institute of Applied Mathematics and Mechanics
of the NAS of Ukraine, G.Batyuka Str. 19, 84100  Sloviansk, Ukraine}
\email{taranets\_r@yahoo.com}

\begin{abstract}
We consider an initial-boundary value problem for a class of nonlocal thin film equations 
governed by the spectral fractional Laplacian with homogeneous Neumann boundary conditions. 
We were the first to establish an $\alpha$-entropy estimate for nonlocal thin film equations, 
which yields essential a priori bounds for the regularity and long-time behavior of weak solutions. 
By developing a localized version of this estimate,  we prove finite speed of propagation, showing that 
the support of nonnegative solutions remains compact for positive times. Furthermore, we find a sufficient condition for a waiting time phenomenon, 
whereby the solution remains identically zero in a region for a nontrivial time interval. 
These results highlight new features in the interaction between nonlocal effects and classical 
thin film dynamics.
\end{abstract}

\maketitle

\section{Introduction}
\label{sec:intro}

In this paper, we discuss various fine qualitative properties of the solutions of the following initial-boundary value problem:
\begin{equation}
\label{eq:ft}
\begin{cases}
\partial_t u(x,t)= \div(u^n(x,t)\nabla p(x,t)), & (x,t) \in  \Omega \times (0,+\infty),\\
p(x,t)= \Ds u(x,t), &  (x,t) \in \Omega \times (0,+\infty),\\
\nabla u\cdot {\hat{\bf{n}}} = u ^{n}  \nabla p(x,t) \cdot{\hat{\bf{n}}}=0, & (x,t) \in \partial\Omega \times (0,+\infty),\\
u(x,0) =u_0(x), & x \in \Omega,
\end{cases}
\end{equation}
where $\Omega \subset \mathbb{R}^d$ is a bounded domain and ${\hat{\bf{n}}}$ is the outward normal to $\partial\Omega$, $n > 0$, and $\Ds$ is the so-called Neumann (spectral) fractional laplacian, defined, for functions in a proper functional framework, as
\begin{align}\label{eq:operator}
\Ds : u\mapsto 
 \mathop {\sum} \limits_{k =0}^{+\infty} {(u,\phi_k)\lambda_k^{s} \phi_k}.
\end{align}
Here, $\{(\lambda_k, \phi_k)\}_{k \in \mathbb{N}}$ are the Neumann eigenpairs (i.e. eigenvalues and eigenfunctions) of the Laplacian operator, and $(\cdot, \cdot)$ is the inner product in $L^2(\Omega)$ (see \cite{AbatangeloValdinoci,StingaTorrea} for further information on this operator).

The evolution above interpolates between the two well-known {\itshape porous medium equation} (PME) (when, formally, $s=0$) and the
{\itshape thin film equation} (TFE) (when, formally, $s=1$). 
This equation arises in the   modelling  of hydraulic fractures: the parameter $s \in (0,1)$ is related to the properties of the medium in which the fracture spreads, while the exponent
$n > 0$ corresponds to different slip conditions on the   fluid-solid interface. In the thin film literature ($s=1$), the case $n \in (1,2)$ corresponds to ``strong slippage'';  $n \in (2,3)$ to ``weak slippage'';  $n = 3$  to a ``no-slip condition'' \cite{RevModPhys.69.931};  $n =2$ to the ``Navier-slip condition'' \cite{JAGER200196}; the case  $n = 1$  arises as the lubrication approximation of the Hele--Shaw flow \cite{GO}.

The PDE in \eqref{eq:ft} is a nonlocal degenerate parabolic-type equation of order $2(s + 1)$. 
As we will see, the properties of the solutions depend on a precise interplay of the values of $n$ and of $s$.  

The first contribution to the mathematical analysis of the system \eqref{eq:ft} is the paper \cite{IM11} by Imbert and Mellet. In \cite{IM11} the authors proved the existence of a non-negative weak solution to \eqref{eq:ft} for $d =1$,
$s = \frac{1}{2}$ and $n \geqslant  1$. This result was generalized for $d =1$, $s \in (0,1)$
and $n \geqslant  1$ in \cite{MR3397309}.
Interestingly, as shown in \cite{IM11,MR3397309}, the value $n = 4$ is critical for \eqref{eq:ft}. Observe that for the 
classical thin film equation the critical $n$ is $n = 3$. 
Later on, for $s = \frac{1}{2}$ and $n\in [1,4)$, in \cite{MR3406645}, Imbert and Mellet were also able to construct self-similar solutions.
In contrast to the one-dimensional situation, the multi-dimensional case is less understood.

The case $n=1$ and $d \geqslant  1$ is considered by J.L. Vazquez and the first author of this paper in \cite{Se_Va}, 
where existence of non-negative solutions to the  Cauchy problem for \eqref{eq:ft} in $\R^d$ is discussed (here 
 $(-\Delta)^s$ denotes the (singular integral) fractional Laplacian of order $s \in (0,1)$ in $\R^d$, see \cite{AbatangeloValdinoci,MR3613319}).
In \cite{Se_Va} explicit compactly supported and non-negative self-similar solutions are constructed. These solutions match the classical Barenblatt profile for the porous medium equation when $s=0$ (see \cite{MR0046217}) and  the Barenblatt-type profile
obtained for zero contact-angle solutions of the TFE in \cite{MR1148286} (for $d=1$) and in \cite{MR1479525} (for $d \geqslant  1$).
These self-similar solutions are then shown to be related to the long-time dynamics under some extra integrability assumption.
The self-similar solutions are constructed among the ones with connected positivity set. Moreover, they are unique in this class.
The uniqueness of the self-similar solutions {\itshape without} assuming that their positivity set is connected has been recently discussed in
\cite{De-Va23}. This uniqueness result is indeed a consequence of the fact, proved in \cite{De-Va23}, that some nonlocal seminorms are strictly decreasing under the continuous Steiner rearrangement.
For our equation (we recall, equation \eqref{eq:ft} in $\R^d$ with $n=1$) the nonlocal seminorm involved is the $H^s(\R^d)$-seminorm, $s\in (0,1)$.
The reason why the $H^s$-seminorm has such a prominent role is that equation \eqref{eq:ft} in $\R^d$ with $n=1$ is
indeed the Wasserstein gradient flow of the $H^s$-seminorm, as recently proved in \cite{Lisini24}.
For \eqref{eq:ft}, with $n \geqslant 1$ and $d=1$, self-similar solutions were constructed in \cite{MR3406645}.

More recently, the system \eqref{eq:ft} was studied by the two authors of the present paper together with De Nitti an Lisini \cite{DNLST24}. The analysis of \cite{DNLST24} is the starting point of the present contribution. 
For $s\in (0,1)$ and $n\in \left[1,\frac{d+2(1-s)}{(d-2s)_{+}}\right)$, we introduced a suitable notion of weak solution for \eqref{eq:ft}. 
Moreover, for $d=2$ and $d=3$ and for $n$ in the smaller interval $(1, \frac{s+2}{s+1})$, we proved the finite speed of propagation (FSP)
property. The finite speed of propagation means that if the initial condition $u_0 \geqslant 0$ is such that, for some $r_0 > 0$,  
\[
\supp u_0 \subseteq B_{r_0}(0),
\]
then there exists a time $T^* >0$ and a   non-decreasing  function $d(t) \in C([0,T^*])$ that verifies
\[
	d(0) = r_0, \qquad d(t) \leqslant   r_0 + C_0 t^{\frac{1}{nd + 2(s+1)}} \qquad\text{for all } t \in [0,T^*],
\]
with the constant $C_0 > 0$ depending on initial data $u_0$ on $d$, on $s$ and on $n$,
such that
	\begin{align*}
	\supp(u(\cdot,t)) \subset  B_{d(t)}(0) \subseteq \Omega \qquad\text{for all } t \in [0,T^*].
	\end{align*}
Besides, within the same range of $n$, and under a suitable control on the initial entropy, 
we also obtained a lower bound for the so-called waiting time phenomenon (WTP). Specifically, 
we showed that there exists a time $T_0 = T_0(n,s, u_0) \leqslant T^*$ such that
\begin{align*}
\supp(u(\cdot,t)) \subseteq B_{r_0}(0) \quad \text{ for } t \in [0,T_0),
\end{align*}
with an exact lower bound of $T_0$ in terms of the initial entropy. 
In \cite{NT2024}, the authors proved FSP and WTP in the   one-dimensional  case and  for all $n \in (1,2)$ by relying on the fact that weak solutions satisfy the improved regularity $u \in L^4((0,T); L^{\infty}(\Omega))$. However, in the multi-dimensional case,
 the available regularity is weaker: weak solutions are only known to belong to $L^2((0,T); L^{\infty}(\Omega))$. 
This is insufficient to control certain terms appearing in the local entropy inequality when the exponent 
$n$ becomes too large relative to $s$, specifically for $n \in [\frac{s+2}{s+1},2) $. 
This limitation explains why, in \cite{DNLST24}, the analysis was restricted to the range $n\in [1, \frac{s+2}{s+1})$.

Note that the range of values for $n$ in both the existence result and the results concerning 
interface propagation (such as finite speed and waiting time phenomena) are not optimal. 
Moreover, they do not align with the optimal range known for the classical thin film equation corresponding to 
$s=1$. In that case, the well-established results indicate that for 
$d=3$, one has $n < 6$ for weak solutions and $n < 4$ for strong solutions (see \cite{DPGG98SIAM}).

The aim of this paper is to extend the validity of the existence result, the finite speed of propagation,
and the lower bound on the waiting time to a broader range of values of $n$ than those considered in \cite{NT2024, DNLST24}.

It should also be noted that the mathematical study of the simplest one-dimensional model  (\ref{eq:ft}), with $s=1$, 
was initiated by Bernis and Friedman in \cite{B8}. They derived a
positivity property and proved the existence of non-negative weak generalized
solutions of initial-boundary problem for non-negative initial data.
Many interesting qualitative properties of the solutions have been discovered and investigated. The finite speed
of propagation (FSP) of the support of the solution was established,  for  $d =1$ and $0< n < 2$ in \cite{BernisFinite} or $2 \leqslant   n < 3$ in \cite{BernisFinite2}, or $n \geqslant  4 $ in \cite{B8}; for $d\in\{2,3\}$ and $\frac{1}{8}< n < 2$ in \cite{ThinViscous} or $2 \leqslant   n < 3$ in \cite{Grun2003}. Sufficient conditions for the WTP were obtained, for $d=1$ and $0 < n < 3$ and for $d \in\{2,3\}$ and $\frac{1}{8} < n < 2$ in \cite{DalPassoGiacomelliGruen} or for $ 2 \leqslant   n < 3$ in \cite{GruenWTWS}. The uniqueness of the solutions to (\ref{eq:ft}) with $s =1$  is an open and challenging problem. Only partial results are available: see \cite{John,MajdoubMasmoudiTayachi,MR2944624}.

To place our work in context, we recall that a key analytical tool in the study of the classical thin film equation 
is the family of so-called $\alpha$-entropy estimates (see, e.g., \cite{DPGG98SIAM}). For TFE ($s=1$), 
these estimates were first established in one spatial dimension ($d=1$) by Beretta, Bertsch, and Dal Passo 
\cite{BerettaBertschDalPasso}, and independently by Bertozzi and Pugh \cite{BP1996}. 
In higher dimensions ($d=2,3$), the estimates were developed by Dal Passo, Garcke, and Gr\"{u}n \cite{DPGG98SIAM}. 
These estimates, which provide bounds on appropriate powers of the solution $u$, are a cornerstone of the theory, 
as they enable a precise control of the admissible regularity of the flux $u^n \nabla (-\Delta) u$. 

For the nonlocal system \eqref{eq:ft}, $\alpha$-entropy estimates were previously unavailable. 
To address this gap and extend the admissible range of the exponent $n$, our first main contribution 
in this paper is the derivation of appropriate $\alpha$-entropy estimates, both in a global and local 
form (see Equation \eqref{eq:alpha-entropy-u1} and Lemma~\ref{lm:add-000}, respectively). 
These estimates are initially established for a carefully chosen family of approximating problems 
and are then used to pass to the limit as the approximation parameters tend to zero.
Following the approach in \cite{DNLST24} (and in line with other works on both classical and nonlocal thin film equations), 
we employ a nested approximation scheme. This involves a suitable regularization of the mobility 
function $u^n$ near zero and at infinity. The advantage of this construction is that the approximate 
solutions remain strictly positive and possess sufficient regularity, which facilitates the derivation and justification of the entropy estimates.

Thanks to these estimates, we are able to establish the analogue of the existence result from \cite{DNLST24} within the extended parameter range 
$n\in \left( \frac{2(s+1)}{ 4(s+1)-d }, \frac{d + 2(1-s)}{d-2s} +\frac{1}{2}\right)$, as well as the corresponding interface 
propagation results in the range $n \in \left(\frac{2(s+1)}{ 4(s+1)-d },2\right) $. The constructed solutions satisfy the so-called 
$\alpha$-entropy estimates and are therefore referred to as {\itshape $\alpha$-entropy solutions}. 
We emphasize that the notion of (weak) solvability adopted here coincides with that used in \cite{DNLST24}, but differs 
from the definition of weak solution in the local case $s=1$ as introduced in \cite{DPGG98SIAM}. 
This difference arises from the absence, in our nonlocal setting, of a precise integration-by-parts formula 
necessary for identifying the flux.

It is worth highlighting that our derivation of the $\alpha$-entropy estimates within the nonlocal framework brings to light an interesting and nontrivial feature: a type of {\itshape nonlocal chain rule formula}. Notably, this formula holds for both ranges 
$s \in (0,1)$ and $s\in (1,2)$ (see Lemma~\ref{lem-GG}).
For the case $s \in (0,1)$, related results have been previously obtained by Buseghin and Garofalo in \cite{BuseGaro}, complementing earlier formulations by Cordoba and Martinez \cite{CM} (see also Cordoba \& Cordoba \cite{CC} for the chain rule for   $(-\Delta)^s(u^2(x))$ in the case $s\in (0,1)$ 
and for the fractional Laplacian on $\R^d$). In this range of $s$, our formula features a different (but equivalent) representation of the remainder term.
The main novelty of our chain rule lies in its extension to the higher-order case $s\in (1,2)$, for which no 
such result was previously available to our knowledge. Roughly speaking, for a smooth function $\phi:\R\to \R$, the formula takes the form: 
$$
(-\Delta)^s \phi(v(x)) = \phi'(v(x))(-\Delta )^s v(x) + \mathscr{R}_s(\phi, v)  \qquad \forall v \in   H^{2s}_N(\Omega),
$$
where the remainder term $\mathscr{R}_s(\phi, v)$ depends on both 
$\phi$ and $v$, as well as their derivatives. When 
$\phi$ is convex, an important difference arises between the cases 
$s\in (0,1]$ and $s\in (1,2)$. In the latter case, the remainder term 
$\mathscr{R}_s(\phi, v)$ does not retain a definite sign, mirroring the behavior observed for 
$s=2$, i.e.  the biharmonic operator. This lack of sign definiteness indicates that the result of 
Cordoba and Martinez \cite{CM} does not extend to higher-order fractional Laplacians.

%

We point out that the main strategy behind our proofs of  FSP  and  WTP  is inspired by the methods developed in
\cite{DalPassoGiacomelliGruen,GruenWTWS, GiacomelliGruen,MR1642176,CT12,MR2073864,MR2265292,MR3989405}.
Our proof of the finite speed of propagation is energetic in nature, relying on a series of various energy estimates 
that allow for the use of a Stampacchia-type iteration lemma. However, due to the nonlocal nature of the operator $(-\Delta)^s$, we cannot apply this strategy directly (indeed, some extra term appears in the local $\alpha$-entropy estimate). 
To resolve this issue, we use  a contradiction-based argument, specifically adapted to the nonlocal setting,
while preserving the essential features of the energetic approach.

\section{Fractional Sobolev spaces and fractional laplacian}
\label{sec:fractional}

In this section, we briefly recall the key definitions and results concerning Sobolev spaces of fractional 
order and the fractional Laplacian that are used throughout this paper. These results are now considered 
classical; for a comprehensive treatment, we refer the reader to \cite{AbatangeloValdinoci, MR2944369}, 
and the references therein.

Let $\{(\lambda_k, \phi_k)\}_{k \geqslant 0}$ denote the sequence of eigenvalues and corresponding 
normalized eigenfunctions of the Laplacian on a bounded, regular domain $\Omega \subset \R^d$, subject to 
homogeneous Neumann boundary conditions. More precisely, each pair $(\lambda_k, \phi_k)$ satisfies: 
\begin{equation}\label{eq:EVN}
\begin{cases}
- \Delta \phi_k = \lambda_k \phi_k  & \text{ in }   \Omega ,\\
\nabla \phi_k \cdot {\hat{\bf{n}}} = 0
 &  \text{ on } \partial\Omega,\\
 \|\phi_k\|_{L^2(\Omega)}=1.
\end{cases}
\end{equation}
Recall that $\lambda_0 = 0$ and that the corresponding eigenfunction is constant, 
given by $\phi_0 = 1/\sqrt{|\Omega|}$. Moreover,
\begin{equation}\label{EVlim}
	\lambda_0=0<\lambda_1\leqslant   \lambda_2\leqslant   \ldots, \qquad \lim_{n\to+\infty}\lambda_n=+\infty.
\end{equation}
For functions $u, v \in L^2(\Omega)$, we denote the inner product in $L^2(\Omega)$ by $(u,v)=\int_\Omega u(x)v(x)\,\d x$.
The sequence ${\phi_k}_{k \geqslant 0}$ forms a Hilbert basis of $L^2(\Omega)$. Accordingly, for any
 $u \in L^2(\Omega)$, we define its Fourier coefficients by
$$
c_k = c_k(u):=(u,\phi_k).
$$ 
We also recall that Parseval’s identity holds:
\begin{equation}\label{Parseval}
	\|u\|^2_{L^2(\Omega)}=\sum_{k=0}^{+\infty} |c_k(u)|^2, \qquad (u,v)=\sum_{k=0}^{+\infty} c_k(u)\,c_k(v).
\end{equation}
We define the Sobolev space $H_N^r(\Omega)$ as 
$$
H_N^r(\Omega):=  \{u\in L^2(\Omega):\,  \|u\|^2_{\dot H_N^r(\Omega)}<+\infty\}.
$$
where, for $r\in [0,+\infty)$,
\begin{equation}\label{defHr}
	\|u\|^2_{\dot H_N^r(\Omega)}:= \mathop{\sum}_{k =0}^{+\infty} {\lambda_k^{r} |c_k|^2}
\end{equation}
is the homogenous $\dot H_N^r(\Omega)$-seminorm. 
Finally, $H_N^r(\Omega)$ is endowed with the norm
\begin{equation}
	\|u\|^2_{H_N^r(\Omega)}:= \|u\|^2_{L^2(\Omega)} + \|u\|^2_{\dot H_N^r(\Omega)}.
\end{equation}

The space $H_N^r(\Omega)$ coincides with the classical fractional Sobolev space $H^r(\Omega)$ for $r\in (0,\frac{3}{2} )$. In the case $r\in (\frac{3}{2},\frac{7}{2})$, we have $H_N^r(\Omega):=  \{u\in H^r(\Omega): \nabla u \cdot {\hat{\bf{n}}} = 0 \text{ on } \partial\Omega\}$. In any case, we have equality of the norms $\|u\|^2_{H_N^r(\Omega)}=\|u\|^2_{H^r(\Omega)}$ for any $u\in H_N^r(\Omega)$.
Moreover, using \eqref{EVlim}, it is not difficult to show that
\begin{equation}\label{eq:equiv}
	\|u\|^2_{L^2(\Omega)} + \|u\|^2_{\dot H_N^r(\Omega)} \quad \text{is equivalent to}\quad \left(\int_\Omega u(x)\,\d x\right)^2 + \|u\|^2_{\dot H_N^r(\Omega)}.
\end{equation}
Using \eqref{EVlim}, it is also immediate to prove that $ H^{r_1}_N(\Omega)\subset H^{r_2}_N(\Omega)$ if $r_1>r_2$. 
Moreover, the following embeddings hold (see \cite{MR2944369}):
\begin{proposition}
Let $r\in [0,+\infty)$. 
\begin{itemize}
	\item If $r<d/2$, then there exists a constant $C$ such that
	\begin{equation}\label{emb1}
 		\| u \|_{L^{2d/(d-2r)}(\Omega)}\leqslant   C  \| u \|_{H^r(\Omega)} \qquad\text{for all } u\in H^r(\Omega).
	\end{equation}
	\item If $r=d/2$ and $p\in (1,+\infty)$, then there exists a constant $C_p$ such that
	\begin{equation}\label{emb2}
 		\| u \|_{L^p(\Omega)}\leqslant   C_p  \| u \|_{H^r(\Omega)} \qquad\text{for all } u\in H^r(\Omega).
	\end{equation}
	\item If $r>d/2$ and $r-d/2\not\in \mathbb{N}$, then there exists a constant $C$ such that
	\begin{equation}\label{emb4}
 	\|u\|_{C^{r-d/2}(\Omega)} \leqslant   C  \| u \|_{H^r(\Omega)} \qquad\text{for all } u\in H^r(\Omega).
	\end{equation}
In particular, 
	\begin{equation}\label{emb3}
 	\| u \|_{L^\infty(\Omega)}\leqslant   C  \| u \|_{H^r(\Omega)} \qquad\text{for all } u\in H^r(\Omega).
	\end{equation}
\end{itemize}
\end{proposition}

The following interpolation inequality follows from the definition of   $\dot{H}_N^r(\Omega)$-norm  and H\"{o}lder's inequality (see \cite{DNLST24} for the proof).
\begin{proposition}\label{prop:interp}
The following interpolation of the semi-norms hold.
If $r_0, r, r_1\in [0,+\infty)$, $r_0 \leqslant   r \leqslant   r_1$ and $u\in H^{r_1}_N(\Omega)$, then
\begin{equation}\label{interpsemi}
 	\| u \|_{\dot H^r_N(\Omega)}\leqslant   \| u \|^{1-\theta}_{\dot H^{r_0}_N(\Omega)}\| u \|^\theta_{\dot H^{r_1}_N(\Omega)},
\end{equation}
where $\theta := \frac{r-r_0}{r_1-r_0}$.
\end{proposition}
%

For $r\in [0,+\infty)$ and $u\in H^{2r}_N(\Omega)$ we define the $r$-Laplacian operator by
\begin{equation}\label{rlap}
(-\Delta)^r u :=   \mathop {\sum}_{k =0}^{+\infty} {  \lambda_k^{r} c_k \phi_k}.
\end{equation}
The following holds:
\begin{proposition}\label{prop:2r}
If $u\in H^{2r}_N(\Omega)$, then
\begin{equation}\label{charL2}
 	\| (-\Delta)^r u \|^2_{L^2(\Omega)} = \| u \|^2_{\dot H^{2r}_N(\Omega)}.
\end{equation}
\end{proposition}

Analogously, using the Parseval product formula in \eqref{Parseval}, it is simple to prove the following integration by parts formula:
\begin{proposition}\label{lem:ip}
Let $r_1, r_2 \in [0,+\infty)$. If $u \in H_N^{r_1+r_2}(\Omega)$ and $v \in H_N^{r_2}(\Omega)$, then
\begin{equation}\label{ip}
	 \int_\Omega((-\Delta)^{r_1} u) ((-\Delta)^{r_2} v)\,\d x =  \int_\Omega((-\Delta)^{r_1+r_2} u )\, v\,\d x.
\end{equation}
\end{proposition}

An immediate consequence of  Propositions \ref{prop:interp} and \ref{prop:2r} 
 is the following interpolation inequality:
\begin{equation}\label{dong-01}
	\| (-\Delta)^{\beta} v \|_{L^2(\Omega)} \leqslant    \| (-\Delta)^{\frac{s+1}{2}} v \|_{L^2(\Omega)}^{\theta} \| v \|_{L^2(\Omega)}^{1-\theta} 
\end{equation}
  is  valid for $v \in   H^{s+1}_N(\Omega)$ and $\beta \in (0, \frac{1+s}{2})$,
where $ \theta := \frac{2\beta}{s+1}$.

We also recall that for $s\in (0,1)$ the fractional Laplacian $(-\Delta)^s$ can be expressed via the heat flow of $-\Delta$ with homogeneous Neumann boundary conditions 
(see, e.\,g. \cite{MR3613319,stinga2019,StingaTorrea}). Indeed, the following holds:
\begin{equation}
\label{eq:semigroup}
(-\Delta)^{s} \psi(x) = \tfrac{1}{\Gamma(-s)} \int \limits_0^{+\infty} { (e^{t\Delta}\psi(x) - \psi(x) ) \tfrac{\d t}{t^{1+s}}},
\end{equation}
where
\begin{equation}
\label{eq:hk}
e^{t\Delta}\psi(x) = \int \limits_{\Omega} { K(x,y,t) \psi(y)\,\d y}, \qquad  e^{t\Delta} 1 = \int \limits_{\Omega} { K(x,y,t)\,\d y} = 1.
\end{equation}
Here, $K(x,y,t)$ is the heat kernel of the heat operator $(\partial_t -\Delta)$ on $\Omega\times (0,+\infty)$ with homogeneous Neumann boundary conditions. It is well known that there exist two positive constants $C_K^{(1)}$ and $C_K^{(2)}$ (see, for example, for Dirichlet boundary conditions
\cite{davies1989heat}  and for Neumann boundary conditions \cite{choulli2015gaussian}) such that
\begin{equation}
\label{control_kernel}
C_K^{(1)} t^{-\frac{d}{2}} e^{-\frac{|x-y|^2}{4t}}\leqslant K(x,y,t)\leqslant C_{K}^{(2)}t^{-\frac{d}{2}} e^{-\frac{|x-y|^2}{4t}}.
\end{equation}
Note that the validity of the representation follows from the general theory, however, for concreteness, we give here a direct proof. 
We denote with $L$ the linear operator defined by the right-hand side of \eqref{eq:semigroup}. 
Given $\psi=\sum_{k=0}^{+\infty}(\psi,\phi_k)\phi_k$ in $H^{2s}_N(\Omega)$, we have that 
\[
L \psi(x) =\sum_{k=0}^{+\infty}(\psi,\phi_k)L\phi_k (x)  \qquad \text{ for } x\in \Omega.
\]
Therefore, we compute $L\phi_k$ via the right-hand side of \eqref{eq:semigroup}. 
To this end, recall that $\phi_k$ is the $k$-th eigenfunction of the Laplacian operator on $\Omega$ with Neumann boundary conditions. 
Thus, it is easy to realize that $v:=e^{t\Delta}\phi_k$, namely, the unique solution of initial-boundary value problem: 
\[
\begin{cases}
	\partial_t v -\Delta v = 0 \qquad &\textrm{ in }\Omega\times (0,+\infty),\\
	\nabla v \cdot {\hat{\bf{n}}} = 0  \qquad &\textrm{ on }\partial\Omega\times (0,+\infty),\\
	v(x,0) = \phi_k(x) \qquad &\textrm{ on }\Omega\times \left\{0\right\},
\end{cases}
\]
is 
\[
e^{t\Delta}\phi_k(x) = e^{-t\lambda_k}\phi_k(x),
\]
where $\lambda_k$ is the eigenvalue associated to 
$\phi_k$. 
As a result,
\[
L\phi_k(x)= \left(\tfrac{1}{\Gamma(-s)}\int_0^{+\infty}\left(e^{-t\lambda_k}-1\right)\tfrac{\d t}{t^{1+s}}\right)\phi_k(x) = \lambda_k^{s}\phi_k(x),
\]
thanks to the well-known formula, which is valid for any $\eta > 0$ and $s \in (0,1)$,
\[
\eta^s=\tfrac{1}{\Gamma(-s)}\int_0^{+\infty}\left(e^{-t\eta}-1\right)\tfrac{\d t}{t^{1+s}}.
\]
Recalling the definition \eqref{rlap} of the fractional Laplacian $(-\Delta)^s$, the formula \eqref{eq:semigroup} easily follows.

\section{Main results: existence and upper bound on interface propagation speed}
\label{sec:main}

We recall from \cite{NT2024,DNLST24} the definition of the convex and lower semicontinuous entropy functional 
$\mathscr{G}:\R\to [0,+\infty]$, characterized by the following properties:
$$
\mathscr{G}''(y) = y^{-n} \text{ for } y>0, \ \ \mathscr{G}'(1)=\mathscr{G}(1)= 0, \text{ and } \mathscr{G} = +\infty \text{ for } y< 0.
$$
Then, the following holds:
\begin{equation}\label{eq:G}
\mathscr{G}(u)=\left\{
\begin{array}{ll}
\displaystyle u \ln u - u +1 & \mbox{ if } n=1, \\[10pt]
\displaystyle  \tfrac{u^{2-n}}{(n-2)(n-1)} +\tfrac{u}{n-1}+ \tfrac{1}{2-n} & \mbox{ if } 1< n< 2, \\[10pt]
\displaystyle \ln \tfrac 1 u +u-1 & \mbox{ if } n=2,\\[8pt]
\displaystyle \tfrac{1}{(n-2)(n-1)}\tfrac{1}{u^{n-2}} +\tfrac{u}{n-1} -\tfrac{1}{n-2} & \mbox{ if } n> 2.
\end{array}
\right.
\end{equation}
Next, we define the $\alpha$-entropy functional $ \mathscr{G}_{\alpha}: [0,+\infty)\to [0,+\infty) $ as the family of functions satisfying
\begin{equation}\label{eq:entropy_prop}
\mathscr{G}_{\alpha}''(z) = z^{\alpha-n}  \qquad \forall \, z\in (0,+\infty).
\end{equation}
Accordingly,
\begin{equation}\label{eq:entropy}
\mathscr{G}_{\alpha}(z): =
\begin{cases} \tfrac{ z^{\alpha  -n +2} - A^{\alpha  -n +2 } }{(\alpha  -n +1)(\alpha  -n +2)} - \tfrac{A^{\alpha  -n +1}}{\alpha  -n +1}(z - A) \,\,&\textrm{
if } \alpha   \neq  n -1, n -2  , \\
z \ln z - z(\ln A + 1) +A \,\,&\textrm{ if }
\alpha   = n - 1 , \\
\ln (\tfrac{A}{z}) + \tfrac{z}{A} - 1, \,\, &\text{ if } \alpha  = n - 2,
\end{cases}
\end{equation}
where $A = 0$ if $\alpha > n-1$ and $A > 0$ if else.
Observe that when $\alpha= 0$, the definition recovers the original entropy functional, i.\,e.
$\mathscr{G}_0 = \mathscr{G}$.
Here is the definition of weak solution. 
\begin{definition}[Weak solution]
\label{def:weak_sol}
A function $u:\Omega\times (0,T)\to \R^+$ is a weak solution of \eqref{eq:ft} if 
\begin{align*}
	u &\in L^\infty((0,T); H^s(\Omega))\cap L^2((0,T);H_N^{1+s}(\Omega))\cap C([0,T];L^2(\Omega)), \\
	\partial_t u &\in L^2((0,T); (W^{1,q}(\Omega))^*),
\end{align*}
for $q := \frac{4d}{2d - n(d-2s)_+}\geqslant  2$, and
\begin{equation}
\label{eq:weak_sol}
\begin{aligned}
& -\iint_{\Omega_T}u\partial_t v \,\d x \,\d t =    \iint_{\Omega_T}  u^n p   \Delta v \,\d x \,\d t
+n   \iint_{\Omega_T} u^{n-1} p  \nabla u  \cdot \nabla v \,\d x \,\d t + \int_{\Omega}u_0v \,\d x, \\
&\qquad\qquad\text{for all }v\in C^{\infty}_{c}(\overline{\Omega}\times [0,T))
\textrm{ such that } \nabla v\cdot \hat{{\bf n}}=0 \textrm{ on }\partial\Omega\times(0,T),\\
&p = (-\Delta )^s u\qquad \textrm{ a.\,e.~in } \Omega_T.
\end{aligned}
\end{equation}

\end{definition}
%

\begin{theorem}[existence of an $\alpha$-entropy solution]
\label{th:ex}
Let 
\[
n\in \left( \tfrac{2(s+1)}{ 4(s+1)-d }, \tfrac{d + 2(1-s)}{(d-2s)_+} +\tfrac{1}{2}\right) , \ \ 
s \in \left(\tfrac{(d-2)_+}{2},1 \right),
\]
and let the initial data $u_0$ satisfy, for some $\alpha \in \left(-1 + \frac{d}{4(s+1) - d},1  \right)$, the condition
\begin{equation}
\label{eq:hyp_u0_ex}
\begin{split}
& u_0\in H^s(\Omega), \,\,\, u_0\geqslant 0, \,\,\, \int_{\Omega}\mathscr{G}_{\alpha}(u_0(x))\d x<+\infty.
\end{split}
\end{equation} 
Then, there exists a weak solution $u$ satisfying
\begin{equation}
\label{eq:alpha-entropy-u1}
\int \limits_{\Omega } \mathscr{G}_{\alpha}(u)   \d x  + \tfrac{4}{ (\alpha +2)^2}  \iint \limits_{\Omega_T} \abs{\Dusm (u^{\frac{\alpha +2}{2}} ) }^2 \d x \d t \leqslant
\int \limits_{\Omega } \mathscr{G}_{\alpha}(u_{0}) \d x  + C_1 \iint \limits_{\Omega_T} { u^{\alpha +2 }  \d x \d t }.
\end{equation}
Moreover, $u$ satisfies 
\begin{eqnarray}\label{eq:ineq}
	\| {u(\cdot,t)}\|_{\overset{.}{H}_N^{s}(\Omega)}^2 + 2 \iint_{\Omega_t} \abs{g(x,r)}^2 \,\d x \,\d r \leqslant   \| u_0 \|_{\overset{.}{H}_N^s(\Omega)}^2  \qquad\text{for all } t\in (0,T],
	\end{eqnarray}
	where the pseudo-flux $g\in L^2(\Omega_T;\R^d)$ is implicitly defined by
	
	\begin{equation}
	\label{eq:pseudo_flux_weak}
	u^{n/2}g = \nabla \left(u^n p\right)-nu^{n-1}p\nabla u\qquad \text{a.\,e.~in }\Omega_T.
	\end{equation}
	Moreover, if 
	$n\in \left( \tfrac{4(s+1)}{ 4(s+1)-d } , \tfrac{d + 2(1-s)}{(d-2s)_+} +\tfrac{1}{2}\right)$ then we can identify $g$ as 
	\begin{equation}
 \label{eq:pseudo_flux}
  { \iint_{\Omega_T}g\cdot  \bm{\phi}    \,\d x \,\d t = -\iint_{\Omega_T} { u^{\frac{n}{2}}p \div  \bm{\phi}  \,\d x \, \d t}
- \tfrac{n}{2}\iint_{\Omega_T} { u^{\frac{n}{2}-1} p \nabla u \cdot \bm{\phi}  \,\d x \,\d t} }
\end{equation}
for all $ \bm{\phi}  \in C_c^\infty ({\Omega}_T;\R^d)$.
\end{theorem}
\begin{remark}
Note that if $n\in \left( \tfrac{4(s+1)}{ 4(s+1)-d } , \tfrac{d + 2(1-s)}{(d-2s)_+} +\tfrac{1}{2}\right)$ then we can identify 
the pseudo-flux $g$ as 
\[
g = 
\begin{cases}
u^{n/2}\nabla p &\qquad \textrm{ on } \,\left\{(x,t) \in \Omega_T: \, u(x,t) >0\right\},\\
0&\qquad \textrm{ on } \left\{(x,t) \in \Omega_T:\, u(x,t)=0\right\}.
\end{cases}
\]
\end{remark}
\begin{theorem}[Upper bounds on interface propagation speed]\label{th:1}
Assume that 
$$
\Omega = B_R(0) , \ \ n \in \left(\tfrac{2(s+1)}{ 4(s+1)-d },2 \right), \ \  s \in \left(\tfrac{(d-2)_+}{2},1 \right) ,
$$
and suppose that
\begin{equation}\label{e-11-000}
	\supp(u_0(\cdot)) \subseteq   B_{r_0}(0)
\end{equation}
for some $r_0 \in (0,R)$. 	Then there exists a time $T^* >0$ and a non-decreasing function $d(t) \in C([0,T^*])$, with $d(0) = r_0$, such that
$$
\supp(u(\cdot,t)) \subseteq B_{d(t)} (0) \subset \Omega  \qquad \text{ for all } t \in [0,T^*],
$$
and
$$
d(t) \leqslant r_0 + C_0 t^{\frac{1}{nd +2(s+1)}}   \qquad \text{ for all } t \in [0,T^*],
$$
where $C_0 > 0$ depends on   $\|u_0\|_1$, $s$, $d$, and $n$. 
\end{theorem}

Finally, we obtain a lower bound on the waiting time.

\begin{theorem}[Lower bounds on waiting times]\label{th:wt}
Let us assume that 
$$
\Omega = B_R(0) , \ \ n \in \left(\tfrac{2(s+1)}{ 4(s+1)-d },2 \right), \ \  s \in \left(\tfrac{(d-2)_+}{2},1 \right) ,
$$
and that condition (\ref{e-11-000}) holds for some $0 < r_0 < R$, and, furthermore, suppose that
\begin{equation}\label{G}
\limsup_{\delta \to 0^+} \, \left( \delta^{-\gamma(\alpha -n+2)} \fint_{B_{r_0}(0)\setminus B_{r_0-\delta}(0)}  \mathscr{G}_{\alpha}(u_0(x))  \dd x\right) < \infty
\end{equation}
for $\gamma \geqslant  \frac{2(s+1)}{n}$.  Let $u$ be a solution of (\ref{eq:ft}) given by Theorem~\ref{th:ex}. Then there exists a time $T_0 = T_0(n,\alpha, s, u_0)$ such that
\begin{align*}
\supp(u(\cdot,t)) \subseteq B_{r_0}(0) \quad \text{ for all } t \in [0,T_0).
\end{align*}
Moreover, the waiting time is estimated from below by
\begin{align*}
T_0 \geqslant  C \left( \sup_{\delta>0} \delta^{-\gamma(\alpha -n+2)} \fint_{B_{r_0}(0)\setminus B_{r_0-\delta}(0)} \mathscr{G}_{\alpha}(u_0(x)) \dd x \right)^{-\frac{n}{\alpha -n+2}}.
\end{align*}
\end{theorem}

As noted in the Introduction, a key ingredient in the proof of the 
$\alpha$-entropy estimates is the following nonlocal chain rule that we prove for Neumann boundary conditions but holds 
also for Dirichlet. We refer the reader 
to Section~\ref{sec:fractional} for the relevant notation.

\begin{lemma}[Nonlocal Chain Rule]
\label{lem-GG}
Let us given a bounded and regular domain $\Omega \subset \mathbb{R}^d$ and a function
$ \phi \in C^{2}(\mathbb{R})$. 

For any function $ u \in H_N^{2\mu}( \Omega )$ and $\mu\in (0,2)$ there holds
\begin{itemize}
\item[(i)]  If  $\mu \in (0,1)$ then
\begin{equation}\label{gg-0-01}
(-\Delta)^{\mu} \phi(u(x)) = \phi'(u(x))(-\Delta)^{\mu} u (x) -  \mathcal{I}_\mu [u](x),
\end{equation}
where
$$
\mathcal{I}_\mu[u](x) := -  \tfrac{1}{ \Gamma(-\mu)} \int_0^{+\infty} { \Bigl(
\int_{\Omega} { K(x,y,t) \int_{u(x)}^{u(y)} {  \phi''(z)   (u(y) - z) \,dz} dy}   \Bigr) \tfrac{dt}{t^{1+\mu}}},
$$
and $\mathcal{I}_\mu [u](x) \geqslant 0$   if $\phi$   is convex.
\item[(ii)] If $\mu = 1$ then
$$
 (-\Delta)  \phi(u(x)) = \phi'(u(x))(-\Delta) u (x) - \phi''(u(x)) |\nabla u(x)|^2.
$$
\item[(iii)] If $\mu \in (1,2) $ then
\begin{equation}\label{gg-0-01-000}
(-\Delta)^{\mu} \phi(u(x)) = \phi'(u(x))(-\Delta)^{\mu} u (x)  -    \mathcal{J}_\mu [u](x),
\end{equation}
where
$$
\mathcal{J}_{\mu } [u](x) := \mathcal{I}_{\mu } [u](x) - \phi''(u(x)) |\nabla u(x)|^2  \tfrac{1}{ \Gamma(1-\mu)}  \int_0^{+\infty} { \Bigl( \int_{\Omega} {  K(x,y,t) dy}   \Bigr) \tfrac{dt}{t^{\mu}}}.
$$
\end{itemize}
\end{lemma}
We postpone the proof of   Lemma~\ref{lem-GG} to the Appendix. When $\phi:\R\to \R$ is $\phi(v) = v^2$ (the original Cordoba \& Cordoba \cite{CC} result),
 we have 
\begin{corollary}
\label{cor:CC}
Under the assumption of Lemma \ref{lem-GG}, if $\mu \in (1,2)$ and $\phi(v) = v^2$, we have
\begin{equation}
\label{eq:lapv2}
\begin{split}
(-\Delta)^{\mu} v^2(x) &= 2v(x) (-\Delta)^{\mu} v (x) +   \tfrac{2 |\nabla v(x)|^2 }{ \Gamma(1-\mu)}  \int _0^{+\infty} { \left( \int_{\Omega} {  K(x,y,t) dy}   \right) \tfrac{dt}{t^{\mu}}} \\     
&+\tfrac{1}{ \Gamma(-\mu)} \int_0^{+\infty} { \left(
\int_{\Omega} { K(x,y,t) (v(y) - v(x))^2 dy}   \right) \tfrac{dt}{t^{1+\mu}}},   
\end{split}
\end{equation}
and 
 \begin{equation}\label{ps-new}
 \begin{split}
\int_{\Omega} {  \abs{(-\Delta)^{\frac{\mu}{2}} v }^2 \, dx   } &=  - \tfrac{1}{ \Gamma(1-\mu)}
\int_{\Omega}{ |\nabla v(x)|^2 \int_0^{+\infty} { \left( \int_{\Omega} {  K(x,y,t) dy}   \right) \tfrac{dt}{t^{\mu}}} dx}  \\
&\,\,-\tfrac{1}{2 \Gamma(-\mu)} \int_{\Omega}{ \int_0^{+\infty} { \left(
\int_{\Omega} { K(x,y,t) (v(y) - v(x))^2 dy}   \right) \tfrac{dt}{t^{1+\mu}}} dx}.
 \end{split}
 \end{equation}
\end{corollary}
Once Lemma~\ref{lem-GG} is established, the proof of \eqref{eq:lapv2} follows immediately. 
The identity \eqref{ps-new} then follows by integration, noting that $\int_\Omega(-\Delta )^{\mu}v^2(x)\d x= 0$,
due to the homogeneous Neumann boundary conditions satisfied by $v$.


\section{Existence of $\alpha$-entropy solutions}

\subsection{A regularized problem}

The core of our results relies on a set of energy and entropy estimates, whose 
formal derivation is presented in this section. These estimates can be rigorously 
justified via the approximation scheme introduced in \cite{DNLST24}, and are therefore 
valid only for those weak solutions that arise as limits of this scheme. Since uniqueness 
of weak solutions is not known, we do not claim that the estimates hold for all possible weak solutions.
Let $s\in \left(\frac{(d-2)_+}{2},1\right)$ and $n\in \left( \tfrac{2(s+1)}{ 4(s+1)-d }, \tfrac{d + 2(1-s)}{(d-2s)_+} +\tfrac{1}{2}\right)$. 
For parameters $s$ and $n$ in the above range, and for initial data 
$u_0$ satisfying
\begin{equation}\label{eq:hyp_u0_ex-0}
u_0\in H^s(\Omega), \,\,\, u_0\geqslant 0, \,\,\, \int_{\Omega}\mathscr{G}(u_0(x))\d x<+\infty,
\end{equation}
as in \cite{DNLST24}, our goal is to construct non-negative weak solutions to \eqref{eq:ft} as limits of the approximate solutions
$u_{\eps,\delta,\gamma}$ to the corresponding regularized problem:
\begin{equation}
\label{eq:approx_old}
\begin{cases}
\partial_t u_{\eps,\delta,\gamma} = \div(m_{\eps,\delta,\gamma}(u_{\eps,\delta,\gamma})\nabla p_{\eps,\delta,\gamma}(x,t)), & (x,t) \in  \Omega \times (0,+\infty),\\
p_{\eps,\delta,\gamma}(x,t)= \Ds u_{\eps,\delta,\gamma}(x,t), &  (x,t) \in \Omega \times (0,+\infty),\\
\nabla u_{\eps,\delta,\gamma}\cdot {\hat{\bf{n}}} = m_{\eps,\delta,\gamma}(u_{\eps,\delta,\gamma})  \nabla p_{\eps,\delta,\gamma}(x,t) \cdot{\hat{\bf{n}}}=0, & (x,t) \in \partial\Omega \times (0,+\infty),\\
u(x,0) =u_{0,\eps,\delta,\gamma}(x), & x \in \Omega,
\end{cases}
\end{equation}
with the approximated mobility:
\begin{equation}
\label{eq:approx_mobility}
m_{\eps,\delta,\gamma}(z) =  m_{\eps,\delta}(z) + \gamma  :=
\begin{cases}
 \frac{z^{n+\beta}}{z^{\beta}+ \eps z^{n}+ \delta z^{n+\beta}}+\gamma, &\qquad z\in [0,+\infty),\\
 \gamma, &\qquad z\leqslant 0,
 \end{cases}
\end{equation}
where $\beta>\max\left\{n, \alpha + 2 \right\}$. 
The parameters $\eps,\, \delta,\, \gamma >0$ are introduced to regularize the degenerate mobility 
and to ensure well-posedness of the approximated problem (\ref{eq:approx_old}).
The parameter $\gamma$ guarantees strict positivity of the mobility,
$$
m_{\eps,\delta,\gamma}(z) \geqslant \gamma \text{ for all } z \in \mathbb{R},
$$
thus removing degeneracy and yielding a uniformly parabolic equation. This allows the use of standard existence 
results and a priori estimates. The parameter $\eps$ regularizes the behaviour of the mobility near $z=0$. Indeed, for small positive values of 
$z$, the term $\eps z^n$ in the denominator of (\ref{eq:approx_mobility}) prevents singular behaviour and ensures 
smoothness of $m_{\eps,\delta }(z)$ at the origin, while preserving the qualitative degeneracy structure of the 
original mobility in the limit $\eps \to 0$. The parameter $\delta$ controls the growth of the mobility for large values of $z$. 
In particular, it provides a uniform upper bound on $m_{\eps,\delta }(z)$, which is essential for obtaining global-in-time 
estimates and compactness. The artificial saturation introduced by $\delta$ is removed by letting $\delta \to 0$.
The assumption $\beta>\max\left\{n, \alpha + 2 \right\}$ ensures sufficient integrability and coercivity properties 
needed to derive uniform energy and entropy estimates and to pass to the limit in the nonlinear terms. The original 
degenerate problem is then recovered by letting $\gamma \to 0$, $\eps \to 0$, and $\delta \to 0$  in a suitable order.

%

We consider an approximation of the non-negative initial data $u_{0}$, chosen such that
\begin{align}
&u_{0, \eps,\delta}\in H^{s+1}(\Omega), && u_{0, \eps,\delta} \geqslant u_{0}+\varepsilon^{\theta_{1}}+\delta^{\theta_{2}}  && \text { for some } 0<\theta_{1}< \tfrac{1}{\beta - \alpha -2},\  \theta_{2}>0, \label{H1p} \\
&u_{0, \eps,\delta,\gamma} \in H^{s+1}(\Omega), && u_{0, \eps,\delta,\gamma} \to u_{0,\eps,\delta} && \text { strongly in } H^{1}(\Omega)  \text { as } \gamma \rightarrow 0,\label{H1c} \\
&{u_{0, \eps,\delta} \in H^{s+1}(\Omega)}, &&  u_{0, \eps,\delta} \to u_{0,\delta} && \text { strongly in } H^{s}(\Omega)  \text { as } \eps  \rightarrow 0, \label{Hsc-0}\\
&{} &&u_{0,\delta} \to u_{0} && \text { strongly in } H^{s}(\Omega)  \text { as } \delta \rightarrow 0, \label{Hsc}
\end{align}
where the parameters \(\varepsilon\) and \(\delta\) are used to lift the initial data to be strictly positive even if \(\gamma=0\).
Therefore, for any $u_0$ satisfying \eqref{eq:hyp_u0_ex} we introduce the following set 
\begin{equation}
\mathscr{A}(u_0):=\left\{u:\Omega\times (0,+\infty)\to \R: \,\,u = \lim_{(\gamma, \eps,\delta)\to (0,0,0)}u_{\eps,\delta,\gamma},\qquad u\textrm{ weak solution of }\eqref{eq:ft}\right\},
\end{equation}
where the notion of weak solvability and the topology in which the approximation is removed are the ones in \cite{DNLST24}.
Note that if $n\in [1,\frac{d+2(1-s)}{(d-2s)_+})$, then this set is non-empty as proved in \cite{DNLST24}.

\subsection{Approximated $\alpha$-entropy}

We also introduce the regularized 
$\alpha$-entropy $\mathscr{G}^{\eps,\delta}_{\alpha}$, defined as a family of smooth, non-negative functions
\begin{equation}\label{eq:alpha-entropy-approx}
 \mathscr{G}^{\eps,\delta}_{\alpha}(z) = \mathscr{G}_{\alpha}(z) + \tfrac{\eps z^{\alpha - \beta + 2}}{(\alpha - \beta +2)(\alpha -\beta +1)} + \tfrac{\delta z^{\alpha+2}}{(\alpha+2)(\alpha +1)}, \text{ where } \beta > \alpha +2, \ \alpha > -1.
\end{equation}
Note that $\mathscr{G}_\alpha^{\eps,\delta}$ is designed in such a way that
\[
(\mathscr{G}_{\alpha}^{\eps,\delta})''(z)m_{\eps,\delta}(z) = z^{\alpha}  \qquad \forall\, z>0.
\]
Moreover,
\[
\begin{split}
& m'_{\eps,\delta}(z) = \tfrac{z^{n+\beta -1}( n z^{\beta} + \eps \beta z^n )}{ ( z^{\beta} + \eps  z^n + \delta z^{n+\beta})^2 } \leqslant
\tfrac{\beta z^{n+\beta -1}}{z^{\beta} + \eps  z^n + \delta z^{n+\beta}}, \\
& (\mathscr{G}_{\alpha}^{\eps,\delta})'(z)=  \mathscr{G}_{\alpha}'(z) + \tfrac{\eps z^{\alpha -\beta +1}}{\alpha - \beta +1}   +
  \tfrac{\delta z^{\alpha  +1}}{\alpha +1}  ,
\end{split}
\]
whence if $\alpha\in \R\setminus \left\{n-1,n-2\right\}$
\begin{equation}
\label{eq:bound_entropy_1}
\abs{\tfrac{m_{\eps,\delta}(z) \mathscr{G}'_{\alpha}(z)}{z^{\alpha+1}} } \leqslant \tfrac{1}{|\alpha -n + 1| }
+ \tfrac{1}{|\alpha - \beta + 1| } + \tfrac{1}{|\alpha  + 1| } + \tfrac{A^{\alpha - n + 1}}{|\alpha -n + 1| } z^{n-\alpha - 1},
\end{equation}
and
\begin{equation}
\label{eq:bound_entropy11}
\abs{\tfrac{m_{\eps,\delta}(z) \mathscr{G}'_{\alpha}(z)}{z^{\alpha+1}}} \leqslant
\begin{cases}
\tfrac{1}{ \beta  -n } + \tfrac{1}{n} +  \abs{\ln(z)} \,\,\,  &\textrm{ if } \alpha = n-1,\\
 \tfrac{1}{|n - \beta - 1| } + \tfrac{1}{|n- 1| } +  1 + z \,\,\, & \textrm{ if } \alpha = n-2,\\
 \end{cases}
\end{equation}
Finally, if $\alpha\in \R\setminus \left\{n-1,n-2\right\}$
\begin{equation}
\label{eq:bound_entropy2}
\abs{\tfrac{m'_{\eps,\delta}(z) \mathscr{G}'_{\alpha}(z)}{z^{\alpha}}} \leqslant \beta \left( \tfrac{1}{|\alpha -n + 1| }
+ \tfrac{1}{|\alpha - \beta + 1| } + \tfrac{1}{|\alpha  + 1| } + \tfrac{A^{\alpha - n + 1}}{|\alpha -n + 1| } z^{n-\alpha - 1} \right),
\end{equation}
and
\begin{equation}
\label{eq:bound_energy21}
\abs{ \tfrac{m'_{\eps,\delta}(z) \mathscr{G}'_{\alpha}(z)}{z^{\alpha}}} \leqslant
\begin{cases}
 \beta \left(
 \tfrac{1}{  \beta -n } + \tfrac{1}{n } + | \ln(z)| \right)  \,\, &\textrm{ if } \alpha = n-1,\\
 \beta \left(
  \tfrac{1}{|n - \beta - 1| } + \tfrac{1}{|n- 1| } + 1 + z \right)  \,\, &\textrm{ if } \alpha = n-2.
  \end{cases}
\end{equation}
In the above displays, $A =0$ if $\alpha > n -1$ while $A = 1$ if $\alpha = n-1$ or $\alpha =n-2$.

%
%



\begin{definition}\label{def-pos}
We say that a couple $(u_{\eps,\delta}, p_{\eps,\delta})$ is a weak solution of (\ref{eq:approx_old}) with $\gamma =0$ if
$u_{\eps,\delta} > 0$ for all $x \in \bar \Omega$ and a.\,e. $t \in [0,T]$,
\[
\begin{aligned}
&(u_{\eps,\delta}, p_{\eps,\delta} ) \in \left(L^{\infty}((0,T); \dot{H}^s(\Omega)) \cap L^{2}((0,T); H_N^{2s+1}(\Omega))\right)
 \times L^{2}((0,T); H^{1}(\Omega)),\\
 &  \partial_t u_{\eps,\delta} \in L^2((0,T); (H^{1}(\Omega))^*),\\
&\lim_{t\to 0^+}u_{\eps,\delta}(\cdot,t) = u_{0,\eps,\delta}(\cdot) \qquad \textrm{ a.\,e.~in } \Omega,
\end{aligned}
\]
and
\begin{equation}\label{apr-001}
\begin{cases}
\displaystyle\int_{0}^T  \langle \partial_t u_{\eps,\delta} , v \rangle_{(H^1)^{*}, H^1} \, \d t & {}
\\ \displaystyle \qquad = -  \iint_{\Omega_T}
  m_{\eps,\delta}(u_{\eps,\delta} )  \nabla p_{\eps,\delta} \cdot  \nabla v  \,\d x\, \d t & \text{for all } v \in L^{2}((0,T); H^{1}(\Omega)),\\
  \displaystyle p_{\eps,\delta} = (-\Delta)^s u_{\eps,\delta} & \textrm{ a.\,e.~in }
   \Omega_T,
   \end{cases}
\end{equation}
\begin{equation}
\label{eq:energy_esti2}
	\norm{u_{\eps,\delta}(\cdot,t)}^{2}_{\dot{H}^s(\Omega)}
	+ 2 \iint_{\Omega_t}  m_{\eps,\delta}(u_{\eps,\delta}) \abs{\nabla p_{\eps,\delta}}^2 \,\d x \,\d \tau
	\leqslant   \norm{u_{0,\eps,\delta}}^2_{\dot{H}^s(\Omega)}
\end{equation}
for all  $t\in [0,T]$.
\end{definition}

 In the paper \cite{DNLST24}, we proved that in the running assumptions for the initial condition $u_0$ (i.e. \eqref{eq:hyp_u0_ex-0}), there exists a weak solution $(u_{\eps,\delta},p_{\eps,\delta})$ that is obtained as limit when $\gamma\to 0$, in a proper topology, of a solution $u_{\eps,\delta,\gamma}$ of \eqref{eq:approx_old}. 
As a consequence, the following set is non-empty
\begin{equation}
\mathscr{A}_{\eps,\delta}(u_0):=\left\{u:\Omega\times (0,+\infty)\to \R: \,\,u =u_{\eps,\delta} = \lim_{\gamma\to 0}u_{\eps,\delta,\gamma},\, u\textrm{ as in Definition \ref{def-pos}}\right\}.
\end{equation}
We also recall that solutions in the class $\mathscr{A}_{\eps,\delta}(u_0)$ satisfy the following estimates,
 which we summarize in the proposition below:
\begin{proposition}
\label{prop:estimates_eps_delta}
Assume that $n\in (  \tfrac{2(s+1)}{ 4(s+1)-d }, \tfrac{d + 2(1-s)}{(d-2s)_+} +\tfrac{1}{2})$ and let $u_0$ verify \eqref{eq:hyp_u0_ex}. Then there exists a positive constant $C$ independent of $\eps$ and $\delta$ such that any element of $ \mathscr{A}_{\eps,\delta}(u_0)$ verify 
\begin{equation}
\label{eq:est_old}
\norm{\partial_t u_{\eps,\delta}}_{L^2(0,T;(W^{1,q}(\Omega))^*)} + \norm{u_{\eps,\delta}}_{L^2(0,T;H^{s+1}(\Omega))} + \norm{p_{\eps,\delta}}_{L^2(0,T;H^{1-s}(\Omega))}\leqslant C,
\end{equation}
where $q:= \frac{4d}{2d - n (d-2s)_+} \geqslant 2$.
\end{proposition}
It is important to note that, although the results in \cite{DNLST24} were originally proven for 
$n\in (1,\frac{d+2(1-s)}{(d-2s)_+})$, this restriction is relevant in \cite{DNLST24} only during the limiting procedure with respect to 
to $\eps$ and $\delta$. In particular, the non-emptiness of the class $\mathscr{A}_{\eps,\delta}$, as well as the validity of Proposition~\ref{prop:estimates_eps_delta}, hold for the broader range $(\frac{2(s+1)}{ 4(s+1)-d }, \frac{d + 2(1-s)}{(d-2s)_+} +\tfrac{1}{2})$. 

\begin{proposition}[extra regularity] \label{ex-reg}
Let $u = u_{\eps,\delta} > 0$ be a weak solution in the sense of Definition \ref{def-pos}.
If $0 < u_{0,\eps,\delta} \in H^{s+1}(\Omega)$ and $s \in ( \tfrac{(d-2)_+}{2},1 )$  then for any $T>0$   we have    
\begin{equation}
\label{eq:extra_regu}
u \in L^{\infty} (0,T; H^{ s+1 }(\Omega)) \cap
  L^{2} (0,T; H^{2(s+1)}(\Omega)).
  \end{equation}
\end{proposition}

\begin{proof}[Proof of Proposition \ref{ex-reg}]
Multiplying (\ref{eq:ft}) by $-\Delta p(u) = (-\Delta)^{s+1} u$, we have
\begin{equation*}
\begin{split}
& \tfrac{1}{2} \tfrac{\d}{\d t}\int_\Omega { |(-\Delta)^{\frac{s+1}{2}} u|^2  \d x }  + \int_\Omega   m_{\eps,\delta}(u) | \Delta p |^2 \d x
  = - \int_\Omega  m'_{\eps,\delta}(u) \nabla u \cdot \nabla p \Delta p(u) \d x  \\
 & \leqslant \tfrac{1}{2} \int_\Omega  m_{\eps,\delta}(u) | \Delta p |^2 \d x + \tfrac{1}{2} \int_\Omega  \tfrac{(m'_{\eps,\delta}(u))^2}{m_{\eps,\delta}(u)} | \nabla u |^2 | \nabla p|^2 \d x,
\end{split}
\end{equation*}
whence
\begin{equation}\label{bs-1}
\tfrac{\d}{\d t}\int_\Omega { |(-\Delta)^{\frac{s+1}{2}} u|^2  \d x }  + \int_\Omega  m_{\eps,\delta}(u) | \Delta p |^2 \d x
\leqslant \int_\Omega  \tfrac{(m'_{\eps,\delta}(u))^2}{m_{\eps,\delta}(u)} | \nabla u |^2 | \nabla p|^2 \d x.
\end{equation}
We discuss separately the two cases $ d< 2s $ and $d\geqslant 2s$. 

In the case $d < 2s$, by the embedding
$H^{s}(\Omega)$ in $C^{s  - \frac{d}{2}}(\bar \Omega)$, from (\ref{bs-1})  we obtain that
\[
\begin{split}
\tfrac{\d}{\d t} \| u \|^2_{H^{s+1}(\Omega)}  + \int_\Omega  m_{\eps,\delta}(u) | \Delta p |^2 \d x
& \leqslant
\mathop {\sup} \limits_{\bar \Omega} ( \tfrac{( m'_{\eps,\delta}(u))^2 }{m_{\eps,\delta}(u)})
\mathop {\sup} \limits_{\bar \Omega} | \nabla u |^2   \int_\Omega    | \nabla p|^2 \d x  \\
&\leqslant \mathop {\sup} \limits_{\bar \Omega} \bigl( \tfrac{  m'_{\eps,\delta}(u)  }{m_{\eps,\delta}(u)}\bigr)^2
\| u \|^2_{H^{s+1}(\Omega)}  \int_\Omega  { m_{\eps,\delta}(u)  | \nabla p|^2 \d x }\\
&  \leqslant
\| u \|^2_{H^{s+1}(\Omega)} (\mathop {\inf} \limits_{\bar \Omega} u)^{-2} \int_\Omega  { m_{\eps,\delta}(u)  | \nabla p|^2 \d x }  ,
\end{split}
\]
whence, using Gr\"{o}nwall's lemma, we arrive at
\begin{equation}\label{bs-2}
\begin{split}
  \| u (t) \|^2_{H^{s+1}(\Omega)} & \leqslant  \| u_{0,\eps,\delta}\|^2_{H^{s+1}(\Omega)} e^{ (\mathop {\inf} \limits_{ \bar \Omega_t} u)^{-2}
\iint_{\Omega_t}  { f_{\eps,\delta}(u)  | \nabla p|^2 \d x \d \tau} } \\
& \leqslant \| u_{0,\eps,\delta}\|^2_{H^{s+1}(\Omega)} e^{ (\mathop {\inf} \limits_{ \bar \Omega_T} u)^{-2}
\| u_{0,\eps,\delta}\|^2_{H^{s}(\Omega)} }
\end{split}
\end{equation}
for all $ t \in [0, T]$.

In the case of $d \geqslant 2s$,  from (\ref{bs-1})  we get
$$
\tfrac{\d}{\d t} \| u \|^2_{H^{s+1}(\Omega)}  + \int_\Omega  m_{\eps,\delta}(u) | \Delta p |^2 \d x
\leqslant
  \tfrac{1}{2} \int_\Omega  m_{\eps,\delta}(u) | \Delta p |^2 \d x  + C  \norm{  \tfrac{m'_{\eps,\delta}(u)}{m^{1/2}_{\eps,\delta}(u)} \nabla u }_{L^{p}(\Omega)}^2 \| \nabla p \|_{L^{q}(\Omega)}^2,
$$
whence
$$
\tfrac{\d}{\d t} \| u \|^2_{H^{s+1}(\Omega)}  + \tfrac{1}{2} \int_\Omega  m_{\eps,\delta}(u) | \Delta p |^2 \d x
\leqslant C  \norm{ \tfrac{m'_{\eps,\delta}(u)}{m^{1/2}_{\eps,\delta}(u)} \nabla u }_{L^{p}(\Omega)}^2 \| \nabla p \|_{L^{q}(\Omega)}^2,
$$
where  $\frac{1}{p}+\frac{1}{q} = \frac{1}{2}$. Note that by standard interpolation results we have
\[
\begin{split}
 \| \nabla p \|_{L^{q}(\Omega)} &\leqslant \|p \|_{W^1_{q}(\Omega)} \leqslant C \|  p \|_{H^{k}(\Omega)} =
C \| u \|_{H^{k+2s}(\Omega)}    \\
&\leqslant C \| u \|^{\theta_0}_{H^{2(s+1)}(\Omega)} \|  u \|^{1-\theta_0}_{H^{ s + 1}(\Omega)} =
C \| \Delta p \|^{\theta_0}_{L^{2}(\Omega)} \| u \|^{1-\theta_0}_{H^{s+1}(\Omega)},
\end{split}
\]
where $ 2 \geqslant k \geqslant 1+ \frac{d}{p} $ with $p \geqslant d$ and $\theta_0 = \frac{k+s-1}{s+1}$. Using these estimates,
we deduce that
\[
\begin{split}
\tfrac{\d}{\d t} \| u \|^2_{H^{s+1}(\Omega)}  &+ C \int_\Omega  m_{\eps,\delta}(u) \abs{ \Delta p }^2 \d x
 \\
&\leqslant C   \mathop {\sup} \limits_{\bar \Omega} \left( \tfrac{  (m'_{\eps,\delta}(u) )^2 }{m^{1+\theta_1}_{\eps,\delta}(u)}\right)
 \| \nabla u \|^{2 }_{L^{p}(\Omega)} \|  u \|^{2(1-\theta_0)}_{H^{ s+1 }(\Omega)}  \| m^{\frac{1}{2} }_{\eps,\delta}(u) \Delta p \|^{2\theta_0}_{L^{2}(\Omega)}.
\end{split}
\]
Therefore, we arrive at
$$
\tfrac{\d}{\d t} \| u \|^2_{H^{s+1}(\Omega)}  + C\int_\Omega  m_{\eps,\delta}(u) | \Delta p |^2 \d x
\leqslant
C\, \varepsilon^{-1}   \mathop {\sup} \limits_{\bar \Omega} (u^{\beta - \frac{2}{1-\theta_0}} )
 \|  \nabla u \|^{\frac{2 }{1-\theta_0}}_{L^{p}(\Omega)} \|  u \|^{2}_{H^{ s+1 }(\Omega)}.
$$
We take $\beta = \frac{2}{1-\theta_0} > 2$, we deduce that
\begin{equation}\label{bs-3}
   \| u (t) \|^2_{H^{s+1}(\Omega)}  \leqslant   \| u_{0,\eps,\delta}\|^2_{H^{s+1}(\Omega)} e^{ C\, \varepsilon^{-1} \int \limits_0^t \|  \nabla u \|^{\frac{2 }{1-\theta_0}}_{L^{p}(\Omega)} \d \tau }
\end{equation}
for all $ t \geqslant 0$.
As $u_{\eps,\delta} \in L^{\infty}(0,T; H^s(\Omega)) \cap L^{2}(0,T; H^{2s+1}(\Omega))$ then by 
the interpolation theory
$$
u_{\eps,\delta} \in   L^{\frac{2}{\vartheta}}(0,T; H^{s+\vartheta(s+1)}(\Omega)), \text{ where } \vartheta \in (0,1).
$$
Since $ L^{\frac{2}{\vartheta}}(0,T; H^{s+\vartheta(s+1)}(\Omega)) \subset L^{\frac{2}{1-\theta_0}}(0,T; W_p^{1}(\Omega))$
provided
$$
\tfrac{1}{s+1} \bigl[ 1-s + \tfrac{d(p-2)}{2p} \bigr] \leqslant \vartheta \leqslant \tfrac{2-k}{s+1} ,
$$
i.\,e.
$$
 d \leqslant p \leqslant \tfrac{d}{k-1 + \frac{d}{2} -s  } \text{ and }  1 +  \tfrac{d}{p} 
 \leqslant  k \leqslant 2 + s -  \tfrac{d}{2} \Rightarrow
 0 <  \tfrac{d}{p}   \leqslant   s -  \tfrac{d-2}{2} 
$$
is true for  $s > \tfrac{(d-2)_+}{2}$, then
$$
\int \limits_0^t \|  \nabla u \|^{\frac{2 }{1-\theta_0}}_{L^{p}(\Omega)} \d \tau \leqslant C(\| u_{0,\eps,\delta}\|_{H^{s}(\Omega)} )
\text{ for all }  t \in [0,T].
$$
As a result, from (\ref{bs-3}) we deduce that
\begin{equation}\label{bs-3-nnn}
   \| u (t) \|^2_{H^{s+1}(\Omega)}  \leqslant   \| u_{0,\eps,\delta}\|^2_{H^{s+1}(\Omega)} e^{  \varepsilon^{-1}C(\| u_{0,\eps,\delta}\|_{H^{s}(\Omega)})}  \text{ for all }  t \in [0,T].
\end{equation}
\end{proof}

We now derive the $\alpha$-entropy estimate at the level of the approximating solutions $u_{\eps,\delta} > 0$, by using Proposition~\ref{ex-reg}.

\begin{proposition}[global $\alpha$-entropy estimate] \label{prop:alpha_entropy_est-new}
Let $u = u_{\eps,\delta} > 0$ be a weak solution in the  class $\mathscr{A}_{\eps,\delta}(u_0)$. 
If $ \alpha \in (-1 + \frac{d}{4(s+1) - d}, 1 ]$ and $s \in ( \frac{(d-2)_+}{2}, 1)$ then there exists a positive constant $C_1 =C_1(d, \alpha, s)$ such that
\begin{equation} \label{eq:alpha_entropy_2_globalt-new}
\int \limits_{\Omega } \mathscr{G}_{\alpha}^{\eps,\delta}(u)   \d x  + \tfrac{4}{ (\alpha +2)^2}  \iint \limits_{\Omega_T} \abs{\Dusm (u^{\frac{\alpha +2}{2}} ) }^2 \d x \d t \leqslant
\int \limits_{\Omega } \mathscr{G}_{\alpha}^{\eps,\delta}(u_{0,\eps,\delta}) \d x  + C_1 \iint \limits_{\Omega_T} { u^{\alpha +2 }  \d x \d t } ,
\end{equation}
where $C_1 = 0$ if $\alpha = 0$.
\end{proposition}

\begin{proof}[Proof of Proposition \ref{prop:alpha_entropy_est-new}.]
For notational convenience, we set  $\mathscr{G}_{\alpha} = \mathscr{G}_{\alpha}^{\eps,\delta}$.  Multiplying (\ref{eq:ft}) by $\mathscr{G}'_{\alpha}(u) $, we have
\begin{equation}\label{eq:conto1-new}
\begin{split}
 \frac{\d}{\d t}\int_\Omega \mathscr{G}_{\alpha}(u)  \d x   & = \int  \mathscr{G}_{\alpha}'(u) \div(m(u)\nabla p) \d x
  =   -\int_\Omega  m(u) \mathscr{G}''_{\alpha}(u)\nabla u\cdot \nabla p\d x \\
 & =  -\int_\Omega   u^{\alpha} \nabla u\cdot \nabla p\d x   =  \tfrac{1}{\alpha + 1}\int_\Omega u^{\alpha +1}   \Delta p\d x \\
 & =
- \tfrac{1}{\alpha + 1}\int_\Omega u^{\alpha +1}  (- \Delta)^{s+1} u \d x .
\end{split}
\end{equation}
%
%
Applying 
Lemma \ref{lem-GG}(iii)  with $\mu = s + 1 \in (1,2)$ and $\phi(z) = z^{\frac{\alpha +2}{2}} $, we obtain
\begin{equation}\label{rrt-ktr}
(- \Delta)^{s+1} u^{\frac{\alpha +2}{2}} = \tfrac{\alpha +2}{2} u^{\frac{\alpha}{2}} (- \Delta)^{s+1} u -
\mathcal{J}_{s+1} [u](x).
\end{equation}
Multiplying this equality by $u^{\frac{\alpha +2}{2}}$ and after then integrating over $\Omega$, we have
$$
\int_\Omega  \bigl| (- \Delta)^{\frac{s+1}{2}} u^{\frac{\alpha +2}{2}} \bigr|^2  \d x =
\tfrac{\alpha +2}{2} \int_\Omega u^{\alpha +1}  (- \Delta)^{s+1} u \, \d x -
\int_\Omega u^{\frac{\alpha + 2}{2}} \mathcal{J}_{s+1} [u] \, \d x .
$$
So, from (\ref{eq:conto1-new}) we arrive at
$$
\frac{\d}{\d t}\int_\Omega \mathscr{G}_{\alpha}(u)  \d x +
\tfrac{2}{(\alpha +1)(\alpha +2)} \int_\Omega  \bigl| (- \Delta)^{\frac{s+1}{2}} u^{\frac{\alpha +2}{2}} \bigr|^2  \d x =
-\tfrac{2}{(\alpha +1)(\alpha +2)} \int_\Omega u^{\frac{\alpha + 2}{2}} \mathcal{J}_{s+1} [u] \, \d x 
$$
provided $\alpha > - 1$.
Using (\ref{ps-new}) with $\mu = s+1 \in (1,2)$ and $v = u^{\frac{\alpha + 2}{2}}$, we have
\[
\begin{split}
\int_\Omega u^{\frac{\alpha + 2}{2}} \mathcal{J}_{s+1} [u] \, \d x &=
- \tfrac{ \alpha  }{  \alpha +2 } \tfrac{1}{ \Gamma( -s)}  \int \limits_{\Omega} { |\nabla u^{\frac{\alpha +2}{2}}(x)|^2  \int \limits_0^{+\infty} { \Bigl( \int \limits_{\Omega} {  K(x,y,t) dy}   \Bigr) \tfrac{dt}{t^{s+1}}} dx} \\
&\quad+ \int_\Omega u^{\frac{\alpha + 2}{2}} \mathcal{I}_{s+1} [u] \, \d x =
\tfrac{\alpha}{ \alpha +2 } \int_\Omega { \bigl| (- \Delta)^{\frac{s+1}{2}} u^{\frac{\alpha +2}{2}} \bigr|^2  \d x}  \\
&\quad + \tfrac{\alpha}{2( \alpha +2)\Gamma(-s-1)}  \int \limits_{\Omega} { \int \limits_0^{+\infty} { \Bigl( \int \limits_{\Omega} {  K(x,y,t) (u^{\frac{\alpha + 2}{2}}(y) - u^{\frac{\alpha + 2}{2}}(x))^2 \,dy}   \Bigr) \tfrac{dt}{t^{s+2}}} dx}  \\
& \quad-\tfrac{\alpha(\alpha +2)}{4\Gamma(-s-1)} \int \limits_\Omega { u^{\frac{\alpha + 2}{2}}(x)
\int \limits_0^{+\infty} { \Bigl( \int \limits_{\Omega} {  K(x,y,t) \int \limits_{u(x)}^{u(y)} {  z^{\frac{\alpha -2}{2}} (u(y) - z) dz} \,dy}   \Bigr) \tfrac{dt}{t^{s+2}}} \, \d x } .
\end{split}
\]
From here, we arrive at
\begin{equation}\label{rt-n-1}
\begin{split}
&\frac{\d}{\d t}\int_\Omega \mathscr{G}_{\alpha}(u)  \d x +
\tfrac{4}{ (\alpha +2)^2} \int_\Omega  \bigl| (- \Delta)^{\frac{s+1}{2}} u^{\frac{\alpha +2}{2}} \bigr|^2  \d x  \\
 & = -  \tfrac{  \alpha }{(\alpha +1)( \alpha +2)^2 \Gamma(-s-1)}    \int \limits_{\Omega} { \int \limits_0^{+\infty} { \Bigl( \int \limits_{\Omega} {  K(x,y,t) F(u(x), u(y)) \,dy}   \Bigr) \tfrac{dt}{t^{s+2}}} dx}.
\end{split}
\end{equation}
Note that
$$
u^{\frac{\alpha + 2}{2}} (x) \int \limits_{u(x)}^{u(y)} { z^{\frac{\alpha -2}{2}} (u(y) - z) \, dz} =
\tfrac{2}{\alpha +2}u^{\alpha +2}(x) + \tfrac{4}{\alpha(\alpha +2) }(u(x) u(y))^{\frac{\alpha +2}{2}} -
\tfrac{2}{\alpha}u^{\alpha +1}(x)u(y),
$$
whence
\begin{equation}
\label{rrt-fff}
\begin{split}
F(u(x), u(y)) & : = ( u^{\frac{\alpha +2}{2} }(y) - u^{\frac{\alpha +2}{2} }(x) )^2 -
 \tfrac{ (\alpha +2)^2}{2} u^{\frac{\alpha + 2}{2}} (x) \int \limits_{u(x)}^{u(y)} { z^{\frac{\alpha -2}{2}} (u(y) - z) \, dz} \\
& =u^{ \alpha +2}(y) - (\alpha +1) u^{\alpha +2}(x) -  \tfrac{4(\alpha +1)}{\alpha}(u(x) u(y))^{\frac{\alpha +2}{2}} +
\tfrac{(\alpha +2)^2}{\alpha}u^{\alpha +1}(x)u(y).
\end{split}
\end{equation}
Note that $F(r,w)$ is continuous function for all $r \geqslant 0$ and $w \geqslant 0$ such that $F(r,r) = 0$
for all $r \geqslant 0$, $\partial_r F(r,w) + \partial_w F(w,r) = 0$ at $r = w$, and $F(\lambda r, \lambda w) = \lambda^{\alpha +2}F(r,w)$
for any $\lambda > 0$.  Taking $r = k$ and $w= k(1+\nu)$ for $\nu \geqslant -1$, we arrive at
$$
F(r,w) = k^{\alpha + 2} \Phi(\nu), \ |r-w|^{\alpha +2} = k^{\alpha + 2} |\nu|^{\alpha + 2},  \text{ where } \Phi(\nu) := F(1,1+\nu),
$$
whence
$$
\frac{F(r,w)}{|r-w|^{\alpha +2}} = \frac{\Phi(\nu)}{|\nu|^{\alpha + 2}}.
$$
From here, we find that
\begin{equation}\label{est-ant}
| F(r, w)| \leqslant  C(\alpha) \,  |r -w |^{\alpha +2} \ \ \ \forall\, r,\,w \geqslant 0
\end{equation}
provided $\alpha \in (-1, 1] $, where
$$
C(\alpha) := \mathop {\sup} \limits_{\nu \geqslant -1} \bigl | \tfrac{\Phi(\nu)}{|\nu|^{\alpha + 2}} \bigr | .
$$
So, we need to show that $0 < C(\alpha) < +\infty$. Really,  
$$
\bigl | \tfrac{\Phi(-1)}{|-1|^{\alpha + 2}} \bigr | = \alpha + 1,  \bigl | \tfrac{\Phi(\nu)}{|\nu|^{\alpha + 2}} \bigr | \to 1 \text{ as } \nu \to +\infty.
$$
If $|\nu| < 1$ then
$$
\Phi(\nu) =   \tfrac{(\alpha +1)(\alpha +2)^2}{12} \nu^3 +
\tfrac{(\alpha +1)(\alpha +2)^2 (\alpha -\frac{4}{3})}{32} \nu^4  +...    
$$
and
$$
\bigl | \tfrac{\Phi(\nu)}{|\nu|^{\alpha + 2}} \bigr | \to 0 \text{ as } \nu \to 0 \text{ if } \alpha < 1,
\ \bigl | \tfrac{\Phi(\nu)}{|\nu|^{\alpha + 2}} \bigr | \to \tfrac{3}{2} \text{ as } \nu \to 0 \text{ if } \alpha = 1.
$$
Therefore, the function $\frac{\Phi(\nu)}{|\nu|^{\alpha + 2}}$ is continuous and bounded for all $\nu \geqslant -1$. 
As a result, we obtain that $ C(\alpha) > 0$ is bounded.

%

Therefore,  using (\ref{est-ant}) and the continuity of $u$ w.r.t. $x$, we get, for any $\eta>0$, 
$$
\abs{\iint_{\{ |x - y| \leqslant \eta \}} { K(x,y,t) F(u(x), u(y)) \d y \d x}} \leqslant
C \iint_{\{ |x - y| \leqslant \eta \}} { K(x,y,t)
| u(y) - u(x) |^{\alpha +2}   \d y \d x}.
$$
Moreover,   by (\ref{control_kernel}) we find that 
$$
\int_0^{+\infty} { t^{-(s+2)} K(x,y,t) \, dt } \leqslant
C  \int_0^{+\infty} { t^{-(s+2 + \frac{d}{2})} e^{-\frac{(x-y)^2}{4t}} \, dt } =
 C  \frac{2^{d+2(s+1)}\Gamma(\frac{d+2(s+1)}{2}) }{|x - y|^{d+2(s+1)}},
$$
and recalling that for $t$ fixed
$u(\cdot, t) \in C^{\vartheta}(\bar \Omega)$ with $\vartheta > \frac{2(s+1)}{\alpha +2}$, we have
\[
\begin{split}
&\abs{\int \limits_0^{+\infty} { \left( \iint_{\{ |x - y| \leqslant \eta \} } {  K(x,y,t) F(u(x), u(y)) \,dy dx }   \right) \tfrac{dt}{t^{s+2}}}}
   \\
\qquad &\qquad \leqslant C \int_{\Omega}{ \int_{\{ \abs{h} \leqslant \eta \}} { \tfrac{| u(x+h) - u(x) |^{\alpha +2} }{h^{d +2(s+1)}}  \d h}  \d x}\\
\qquad &\qquad \leqslant  C \, \eta^{ \vartheta(\alpha+2) - 2(s+1)}.
\end{split}
\]
Thanks to Proposition \ref{ex-reg}, we know that 
$0 < u_{0,\varepsilon,\delta} \in H^{ s+1 }(\Omega)$ implies
$0 < u_{\varepsilon,\delta} \in L^{\infty} (0,T; H^{ s+1}(\Omega)) \cap L^{2} (0,T; H^{2(s+1)}(\Omega)) $. In particular, due to the
compact embedding $H^{2(s+1)}(\Omega)$ in $C^{2(s+1) - \frac{d}{2}}(\bar \Omega)$, we have
$u_{\varepsilon,\delta}(\cdot, t) \in C^{2(s+1) - \frac{d}{2}}(\bar \Omega)$ for a.\,e. $t > 0$.
We fix $ \vartheta:= 2(s+1) - \frac{d}{2}$ and we observe that in the current range for $\alpha$ (in particular, $\alpha > -1 + \frac{d}{4(s+1) - d}$)
we have that $\vartheta > \frac{2(s+1)}{\alpha +2} $. Therefore, 
we arrive at
\begin{equation}\label{ant-sym}
\mathop {\lim}_{\eta \to 0} \int_0^{+\infty} { \left( \iint_{\{ \abs{x - y} \leqslant \eta \} } {  K(x,y,t) F(u(x), u(y)) \,dy dx }   \right) \tfrac{dt}{t^{s+2}}} = 0.
\end{equation}
Consequently, we conclude with the help of (\ref{rt-n-1}) and of (\ref{ant-sym}) that
\begin{equation}\label{rt-n-2}
\frac{\d}{\d t}\int_\Omega \mathscr{G}_{\alpha}(u)  \d x +
\tfrac{4}{ (\alpha +2)^2} \int_\Omega  \bigl| (- \Delta)^{\frac{s+1}{2}} u^{\frac{\alpha +2}{2}} \bigr|^2  \d x \leqslant
 C \int_{\Omega} { u^{\alpha + 2} \, \d x }.
\end{equation}
Integrating (\ref{rt-n-2}) in time, we obtain (\ref{eq:alpha_entropy_2_globalt-new}).
\end{proof}

Next, we state and prove the following {\itshape $\alpha$-entropy a priori estimate}, which is a consequence of 
Proposition \ref{prop:alpha_entropy_est-new}.

\begin{proposition}[$\alpha$-entropy a priori estimate]
\label{prop:alpha_entropy_est}
Assume that $s \in \left(\frac{(d-2)_+}{2}, 1 \right) $. Let $u = u_{\eps,\delta} > 0$ be a weak solution in the  class $\mathscr{A}_{\eps,\delta}(u_0)$.
The following inequalities hold
\begin{itemize}
\item If $\max\{ -1 + \frac{d}{4(s+1) - d}, n -2\} < \alpha \leqslant  1$ then there exist  a positive constant
$C_g=C_g(d,\Omega, \alpha, n)$ and $T^* > 0$ such that for any $T \in [0,T^*]$ there holds
\begin{equation}
\label{eq:alpha_entropy_est2}
\begin{split}
 \int_{\Omega } \mathscr{G}_{\alpha}^{\eps,\delta} (u(T))   \d x  + \tfrac{2}{ (\alpha +2)^2 } \iint_{\Omega_T} \abs{\Dusm (u^{\frac{\alpha +2}{2}} ) }^2 \d x \d t
 \leqslant
C_g \int_{\Omega } \mathscr{G}_{\alpha}^{\eps,\delta}(u_{0,\eps,\delta}) \d x;
 \end{split}
 \end{equation}
\item If $-1 + \frac{d}{4(s+1) - d} < \alpha \leqslant   1$ 
then  for any $T>0$
 \begin{equation} \label{eq:alpha_entropy_2_globalt}
 \int_{\Omega } \mathscr{G}_{\alpha}^{\eps,\delta} (u)   \d x  + \tfrac{4}{ (\alpha +2)^2 } \iint_{\Omega_T} \abs{\Dusm (u^{\frac{\alpha +2}{2}} ) }^2 \d x \d t \\
 \leqslant
\int_{\Omega }\mathscr{G}_{\alpha}^{\eps,\delta} (u_{0,\eps,\delta}) \d x + C\,T\, \norm{u_{0,\eps,\delta}}_{\dot{H}^s(\Omega)}^{\alpha+2}.
 \end{equation}
 \end{itemize}
\end{proposition}
\begin{proof}[Proof of Proposition \ref{prop:alpha_entropy_est}.]
We work within the approximation scheme, assuming 
$u=u_{\eps,\delta} > 0$, and make use of estimate~(\ref{eq:alpha_entropy_2_globalt-new}).

We begin by proving \eqref{eq:alpha_entropy_2_globalt}. To this end, we recall that, under 
the current assumptions, the weak solutions constructed satisfy the energy dissipation estimate 
(see \cite[Theorem 3.1]{DNLST24}):
 $$
 \norm{u(t)}_{\dot{H}^s(\Omega)}\leqslant \norm{u_{0,\eps,\delta}}_{\dot{H^s}(\Omega)}.
 $$
 Therefore, thanks to the Sobolev embeddings, we have that $u\in L^\infty(0,T;L^{\frac{2d}{d-2s}}(\Omega))$ and thus
 $u^{\alpha+2}\in L^1(\Omega)$ whenever  $0< \alpha +2 \leqslant \frac{2d}{d-2s}$, namely,  $-2 < \alpha \leqslant \frac{4s}{d-2s}$.
 As a result, we have that in this range of $\alpha$ (and of $n$) there exists a constant $C=C(\Omega,d,s)$ such that
$$
\mathop {\textrm{ess sup}} \limits_{t\in (0,T)} \int_{\Omega} u^{\alpha +2}\d x \leqslant C\norm{u_{0,\eps,\delta}}_{\dot{H}^s(\Omega)}^{\alpha+2},
$$
and then from (\ref{eq:alpha_entropy_2_globalt-new}) we get (\ref{eq:alpha_entropy_2_globalt}).

Next, we will prove (\ref{eq:alpha_entropy_est2}). To absorb the right-hand side of (\ref{eq:alpha_entropy_2_globalt-new}), i.\,e. $\int_\Omega u^{\alpha +2}\d x$, with the terms in the left-hand side of the one, we are forced to restrict the time interval. Thus, the resulting estimate will be only {\emph {local} } in time.

Applying Lemma~\ref{G-N-nn} to $w:= u^{\frac{\alpha +2}{2}}$ with $b = \frac{2(\alpha - n+2)}{\alpha+2}$,
$\theta = \frac{n d}{n d + 2(s+1)(\alpha - n +2)}$ and  $\alpha > n -2$, thanks to  Young's inequality in the form:
\[
a b \leqslant \eps a^p + C(\eps)b^{p'}  \qquad \forall \, a, b>0, \qquad  \tfrac{1}{p}+ \tfrac{1}{p'} = 1,
\] 
we have, with obvious choices of $a, b$ and $\eps, p$,
\[
\begin{split}
 \int_{\Omega}u^{\alpha+2}\d x = \int_\Omega w^2 \d x & = \norm{w}_{L^2(\Omega)}^2
  \leqslant C \| \Dusm w \|_{L^2(\Omega)}^{2\theta} \|w\|^{2(1 -\theta)}_{L^{\frac{2(\alpha - n+2)}{\alpha+2}}(\Omega)}
+ C \|w\|_{L^{\frac{2(\alpha - n+2)}{\alpha+2}}(\Omega)}^2 \\
& \quad\leqslant \tfrac{2}{ (\alpha +2)^2 } \norm{\Dusm\left(u^{\frac{\alpha+2}{2}}\right)}^2_{L^2(\Omega)} + C\left(\int_{\Omega}u^{\alpha-n+2}\d x\right)^{\frac{\alpha +2}{\alpha-n+2}}.
\end{split}
\]
As a result, (\ref{eq:alpha_entropy_2_globalt-new}) becomes
\begin{equation}\label{eq:entropy_est_prelim}
\begin{split}
\int_\Omega \mathscr{G}_{\alpha}^{\eps,\delta}(u(t)) \d x     + \tfrac{2}{ (\alpha +2)^2 } \iint_{\Omega_t} \abs{\Dusm (u^{\frac{\alpha +2}{2}} ) }^2   \d x \d \tau    \\
\qquad \qquad \qquad \leqslant \int_\Omega \mathscr{G}_{\alpha}^{\eps,\delta}(u_{0,\eps,\delta}) \d x + C \int_0^t { \left( \int_\Omega \mathscr{G}^{\eps,\delta}_{\alpha}(u)   \d x  \right)^{\frac{\alpha +2}{\alpha - n + 2}} \d \tau}.
\end{split}
\end{equation}
Observe that since $\alpha >n -2$ by assumption, then $\frac{\alpha + 2}{\alpha-n +2}>1$. Therefore,
Gr\"{o}nwall's lemma gives
$$
\int_{\Omega} \mathscr{G}_{\alpha}^{\eps,\delta}(u)\d x \leqslant 2^{\frac{\alpha-n+2}{n}}\int_{\Omega}\mathscr{G}^{\eps,\delta}_{\alpha}(u_{0,\eps,\delta})\d x 
$$
for any $t\in [0, T^*]$, where
$$
T^*:= \frac{ \alpha-n+2 }{2 C n \left(\int_\Omega \mathscr{G}_{\alpha}^{\eps,\delta}(u_{0,\eps,\delta})\d x\right)^{\frac{n}{\alpha-n+2}}}.
$$
On the same time interval
\[
\begin{split}
\int_{0}^t \left(\int_{\Omega}\mathscr{G}_{\alpha}^{\eps,\delta}(u(\tau))\d x\right)^{\frac{\alpha + 2}{\alpha-n +2}}\d \tau & \leqslant 2^{\frac{\alpha +2}{n}}T^{*}\left(\int_{\Omega}\mathscr{G}_{\alpha}^{\eps,\delta}(u_{0,\eps,\delta})\d x\right)^{\frac{\alpha + 2}{\alpha-n +2}}\\
&\quad =\tfrac{2^{\frac{\alpha-n+2}{n}}(\alpha-n+2) }{C\, n} \int_{\Omega}\mathscr{G}_{\alpha}^{\eps,\delta}(u_{0,\eps,\delta})\d x.
\end{split}
\]
Therefore, from \eqref{eq:entropy_est_prelim} we arrive at (\ref{eq:alpha_entropy_est2}) with
$$
C_g =   \tfrac{2^{\frac{\alpha-n+2}{n}}(\alpha-n+2) }{C\, n} + 1 .
$$
\end{proof}

\subsection{Proof of the Theorem \ref{th:ex}: Existence of $\alpha$-entropy solutions.}

We refine our analysis by combining the uniform estimate from Proposition \ref{prop:estimates_eps_delta}, 
as given in \eqref{eq:est_old}, with the $\alpha$-entropy estimate established in Proposition \ref{prop:alpha_entropy_est}.

We recall that the exponent $n$ lies in the range
\begin{equation}
\label{eq:interval_n}
n\in ( \tfrac{2(s+1)}{ 4(s+1)-d }, \tfrac{ d + 2(1-s)}{(d-2s)_+} + \tfrac{1}{2}),
\end{equation}
and that estimate  \eqref{eq:alpha_entropy_2_globalt} holds for any $\alpha$ such that 
\begin{equation}
\label{eq:rangealpha}
\alpha \in \left(-1 + \tfrac{d}{4(s+1) - d} ,1 \right].  
\end{equation} 
%
%
We recall that $u_{\eps,\delta}$ verifies 
\begin{equation}
\label{eq:weak_eps_delta}
\int_{0}^T  \langle \partial_t u_{\eps,\delta} , v \rangle_{(W^{1,q}(\Omega))^*, W^{1,q}(\Omega)} \, \d t = \iint_{\Omega_T} J(u_{\eps,\delta})\cdot\nabla v\,\d x \,\d t ,
\end{equation}
where $J_{\eps,\delta}$ is defined weakly as 
\begin{equation}
\label{eq:flux_eps_delta}
\iint_{\Omega_T} J_{\eps,\delta}(u_{\eps,\delta})\cdot {\bf V} \d x \d t = -\iint_{\Omega_T} p_{\eps,\delta}\nabla(m_{\eps,\delta}(u_{\eps,\delta}))\cdot {\bf V} \d x\d t - \iint_{\Omega_T} p_{\eps,\delta}m_{\eps,\delta}(u_{\eps,\delta})\textrm{div} {\bf V} \d x \d t 
\end{equation}
for any ${\bf V} \in C^{\infty}(\Omega_T)$.
Thanks to Proposition \ref{prop:estimates_eps_delta} and standard compactness results, there exist functions
$u$ and  $p = (-\Delta)^s u$, along with a (non-relabelled) subsequence of $\eps,\delta$ such that  
\begin{equation}
\label{eq:conv_eps_delta}
\begin{split}
& u_{\eps,\delta}\xrightarrow{\eps,\delta \to 0}u \quad \textrm{ strongly in } C^{0}([0,T];L^p(\Omega))\,\,\,\,\forall p< p_s := \tfrac{2d}{d-2s},\\
& \nabla u_{\eps,\delta} \xrightarrow{\eps,\delta \to 0}\nabla u \quad \textrm{ strongly in } L^2(0,T;H^{r_0}(\Omega))\,\,\forall r_0<s,\\
& p_{\eps,\delta}\xrightarrow{\eps,\delta \to 0} p \quad \textrm{ strongly in } L^2(0,T; H^{1-r_0}(\Omega))\,\,\forall r_0\in (s,1].
\end{split}
\end{equation}
Since 
\[
\int_{0}^T  \langle \partial_t u_{\eps,\delta} , v \rangle_{(W^{1,q})^{*}, W^{1,q}} \, \d t  = -\iint_{\Omega_T}u_{\eps,\delta}\partial_t v \,\d x,\d t -\int_{\Omega}u_{0,\eps,\delta} v(x,0)\,\d x, 
\]
we easily obtain that 
\begin{equation}
\label{eq:limit_lhs}
\lim_{(\eps,\delta)\to (0,0)}\int_{0}^T  \langle \partial_t u_{\eps,\delta} , v \rangle_{(W^{1,q})^{*}, W^{1,q}} \, \d t = -\iint_{\Omega_T}u\,\partial_t v \,\d x\d t -\int_{\Omega}u_0(x) v(x,0)\,\d x.
\end{equation}
Now, we pass to the limit in the r.h.s. of \eqref{eq:weak_eps_delta}, namely in 
\[
\label{eq:rhs_ex}
\begin{split}
-\iint_{\Omega_T} J(u_{\eps,\delta})\cdot\nabla v\d x \d t &= \iint_{\Omega_T} p_{\eps,\delta}\nabla(m_{\eps,\delta}(u_{\eps,\delta}))\cdot \nabla v\d x \d t + \iint_{\Omega_T} p_{\eps,\delta}m_{\eps,\delta}(u_{\eps,\delta})\Delta v \d t \d x \\
&= I_{1}(\eps,\delta) + I_2(\eps,\delta).
\end{split}
\]
To this end note that possibly extracting a further subsequence, we can assume that 
$u_{\eps,\delta}	\to u$, $\nabla u_{\eps,\delta}\to \nabla u$ and $p_{\eps,\delta}\to p$ almost everywhere in $\Omega_T$. Therefore, 
\begin{equation}
\label{eq:ae}
\begin{split}
&p_{\eps,\delta} \nabla( m_{\eps,\delta}(u_{\eps,\delta}))\cdot \nabla v \xrightarrow{(\eps,\delta)\to (0,0) }n\,p \,u^{n-1} \nabla u\cdot \nabla v\qquad \textrm{a.e. in }\Omega_T,\\
& p_{\eps,\delta}m_{\eps,\delta}(u_{\eps,\delta})\Delta v \xrightarrow{(\eps,\delta)\to (0,0) } p \,u^n \Delta v\qquad \textrm{a.e. in }\Omega_T.
\end{split}
\end{equation}
Recall that
\begin{equation}
\label{eq:m'}
m_{\eps,\delta}'(y) = \frac{y^{-1}(n y^{\beta} + \varepsilon \beta y^n)}{y^{\beta} + \varepsilon y^n + \delta y^{n+\beta}}m_{\eps,\delta}(y)
\leqslant \beta \, y^{-1} m_{\eps,\delta}(y) \leqslant \beta \, y^{n-1},
\end{equation}
and that 
the $\alpha$-entropy estimate allows to control
\[
\iint_{\Omega_T}\abs{(-\Delta)^{\frac{1+s}{2} }u_{\eps,\delta}^{\frac{\alpha+2}{2}}}^2 = \int_{0}^T \norm{u^{\frac{\alpha+2}{2}}_{\eps,\delta}}_{H^{1+s}(\Omega)}^2 \d t = \int_{0}^T\norm{\nabla u_{\eps,\delta}^{\frac{\alpha+2}{2}}}^2_{\dot H^s(\Omega)} \d t,
\]
then
\begin{equation}
\label{eq:entro1}
\norm{\nabla u_{\eps,\delta}^{\frac{\alpha+2}{2}}}_{L^2(0,T;L^{p_s}(\Omega))}\leqslant C.
\end{equation}
Thus, we rewrite
\[
\begin{split}
\nabla m_{\eps,\delta}(u_{\eps,\delta})&= \tfrac{2}{\alpha+2}m_{\eps,\delta}'(u_{\eps,\delta}) u_{\eps,\delta}^{-\frac{\alpha}{2}}\nabla (u_{\eps,\delta}^{\frac{\alpha+2}{2}})\\
& = \tfrac{2}{\alpha+2} u_{\eps,\delta}^{n-1-\frac{\alpha}{2}} \tfrac{m'_{\eps,\delta}(u_{\eps,\delta})}{u^{n-1}}\nabla (u_{\eps,\delta}^{\frac{\alpha+2}{2}}),
\end{split}
\]
and note that since $n >\tfrac{2(s+1)}{4(s+1)-d}=\tfrac{1}{2} + \tfrac{d}{2(4(s+1)-d)}$ we have that $n-1-\frac{\alpha}{2}   >  0$. 
The term $I_1(\eps,\delta)$ can be rewritten as 
\[
\begin{split}
I_{1}(\eps,\delta) &= \tfrac{2}{\alpha+2} \iint_{\Omega_T} {p_{\eps,\delta} \, u_{\eps,\delta}^{n-1-\frac{\alpha}{2}}\tfrac{m'_{\eps,\delta}(u_{\eps,\delta})}{u_{\eps,\delta}^{n-1}}\nabla (u_{\eps,\delta}^{\frac{\alpha+2}{2}})\cdot \nabla v \, \d x \d t}\\
& \leqslant C\norm{p_{\eps,\delta}}_{L^2(0,T;L^{r'}(\Omega))}\norm{\nabla m_{\eps,\delta}(u_{\eps,\delta})}_{L^2(0,T;L^{r}(\Omega))},
\end{split}
\]
where $r = \frac{2d}{d + 2(1 -r_0)}$ and $r' = \frac{2d}{d - 2(1 -r_0)}$ is the Sobolev immersion exponent of $H^{1-r_0}(\Omega) \hookrightarrow L^r(\Omega)$. 
Then we take $r_1 = p_s $ and  $r_2 = \frac{d}{1+s-r_0}$ in such a way that $\frac{1}{r} = \frac{1}{r_1} + \frac{1}{r_2}$. 
To have a useful estimate, we need that 
\begin{equation}
\label{eq:needed}
r_2 (n-1-\tfrac{\alpha}{2}) \leqslant  p_s.
\end{equation}
 As $n<\tfrac{d + 2(1-s)}{(d-2s)_+} + \frac{1}{2}$ by assumption, we can tune $\alpha$ in \eqref{eq:rangealpha} in such a way that $1 + \frac{\alpha}{2}  < n \leqslant 1 + \frac{\alpha}{2} + \frac{2(1+s-r_0)}{(d-2s)_+} < \frac{d+2(1-s)}{(d-2s)_+}+ \frac{\alpha}{2}$, which implies \eqref{eq:needed}.

Therefore, due to \eqref{eq:m'}, \eqref{eq:entro1} and  \eqref{eq:conv_eps_delta}, we get that there exists a constant $C$ independent of $\eps,\delta$ such that 
\[
\begin{split}
\int_{0}^T \| \nabla m_{\eps,\delta}(u_{\eps,\delta})\|^2_{L^{r}(\Omega)}  \d t &\leqslant
C \int_{0}^T \|  u_{\eps,\delta}^{n-1-\frac{\alpha}{2}} \nabla (u_{\eps,\delta}^{\frac{\alpha+2}{2}})\|^2_{L^{r}(\Omega)}  \d t
  \\
&  \leqslant C \int_0^T \| \nabla (u_{\eps,\delta}^{\frac{\alpha+2}{2}}) \|^2_{L^{r_1}(\Omega)} \| u_{\eps,\delta}^{n-1-\frac{\alpha}{2}} \|^2_{L^{r_2}(\Omega)} \d t  \leqslant C.
\end{split}
\]
%
%
%
%
%
%
%
Consequently, we have that along a further subsequence
\[
 \nabla m_{\eps,\delta}(u_{\eps,\delta})\xrightarrow{(\eps,\delta) \to (0,0)} \mathfrak{V} \quad \textrm{ weakly in } L^2(0,T; L^{r}(\Omega )).
\] 
On the other hand, since $ \nabla m_{\eps,\delta}(u_{\eps,\delta})\xrightarrow{(\eps,\delta) \to (0,0)}  n u^{n-1} \nabla u$ almost everywhere in $\Omega_T$, 
we identify $\mathfrak{V}= n u^{n-1} \nabla u$, i.\,e.
\begin{equation}\label{rrt-1}
  \nabla m_{\eps,\delta}(u_{\eps,\delta})\xrightarrow{(\eps,\delta) \to (0,0)} n u^{n-1} \nabla u \quad \textrm{ weakly in } L^2(0,T; L^{r}(\Omega )).
\end{equation}
Now, since \eqref{eq:conv_eps_delta} implies that 
\begin{equation}\label{rrt-2}
 p_{\eps,\delta}\xrightarrow{\eps,\delta \to 0} p \quad \textrm{ strongly in } L^2(0,T; L^{r'}(\Omega)), \ r' = \tfrac{2d}{d - 2(1 -r_0)}, 
\end{equation}
we conclude as $I_{1}(\eps,\delta)$ contains the product of a weakly converging sequence and of a strongly convergent sequence. Therefore,
%
%
\[
\lim_{(\eps,\delta)\to (0,0)}I_{1}(\eps,\delta) = n  \iint_{\Omega_T}p  \, u^{n-1} \nabla u \cdot \nabla v\,\d x \d t. 
\]
The identification of $\lim_{(\eps,\delta)\to (0,0)}I_{2}(\eps,\delta)$ proceeds in a similar way. Indeed, we recall that 
$m_{\eps,\delta}(u_{\eps,\delta})\leqslant   \abs{u_{\eps,\delta}}^n$ almost everywhere in $\Omega_T$. 
Therefore,
$$
\abs{ m_{\eps,\delta}(u_{\eps,\delta}) p_{\eps,\delta} }  \leqslant   
  \abs{u_{\eps,\delta}}^{n}  \abs{p_{\eps,\delta}}.   
$$
As $u_{\eps,\delta} \in L^\infty (0,T; L^{p_s}(\Omega)) \cap L^{\alpha + 2} (0,T; L^{\infty}(\Omega))$ then by space-time interpolation 
we have
\begin{equation}\label{rrt-3}
u_{\eps,\delta} \in L^q (0,T; L^{p}(\Omega)) \text{ with } p = \tfrac{p_s q}{ q-\alpha -2 }, \ q > \alpha +2.
\end{equation}
We choose $q = 2n$. Clearly, $q>\alpha+2$ as the range for $n$ implies that $n> \frac{\alpha}{2}  +1$. 
Then we fix $p = n r$ with $r = \frac{2d}{d + 2(r_0-1)}$, and therefore we need that 
\[
p = nr \leqslant \tfrac{p_s q}{ q-\alpha -2 },
\]
namely, 
\[
n -\tfrac{\alpha}{2} -1 \leqslant \tfrac{d + 2(1-r_0)}{d-2s}.
\]
As before, this inequality is satisfied thanks to \eqref{eq:interval_n} and to \eqref{eq:rangealpha}.
Therefore,
$$
\int_{0}^T \|   m_{\eps,\delta}(u_{\eps,\delta})   \|^2_{L^{r}(\Omega)}  \d t \leqslant
\int_{0}^T \|  u_{\eps,\delta}  \|^{2n}_{L^{ nr}(\Omega)}  \d t
\leqslant    C.
$$
%
%
%
%
%
%
%
As a result,
\begin{equation}\label{rrt-4}
 m_{\eps,\delta}(u_{\eps,\delta})   \xrightarrow{(\eps,\delta) \to (0,0)}  u^{n}  \quad \textrm{ weakly in } L^2(0,T; L^{r}(\Omega )) .
\end{equation}
Then, since   
$ m_{\eps,\delta}(u_{\eps,\delta}) p_{\eps,\delta} \xrightarrow{(\eps,\delta)\to (0,0)}u^n p $ almost everywhere in $\Omega_T$, due to 
(\ref{rrt-2}) and (\ref{rrt-4}),
using the (generalized) Lebesgue Theorem of dominated convergence, we conclude
that 
\[
\lim_{(\eps,\delta)\to (0,0)}I_{2}(\eps,\delta) = \iint_{\Omega_T} u^n p \Delta v\, \d x \d t.
\]
Summing up, we have that $(u,p)$ satisfies
\[
 -\iint_{\Omega_T}u\,\partial_t v \,\d x\d t -\int_{\Omega}u_0(x) v(x,0)\,\d x = n \iint_{\Omega_T}p\, u^{n-1} \nabla u \cdot \nabla v\d x \d t +  \iint_{\Omega_T} p\, u^n \Delta v \d x \d t,
\]
namely, $u$ is a weak solution.

\subsubsection{Limits $\eps,\delta \to 0$ in the $\alpha$-entropy and energy estimate}
\label{ssec:limit_entropy_energy}
In this subsection, we show that the weak solution constructed above satisfies the $\alpha$-entropy estimate \eqref{eq:alpha-entropy-u1} 
as well as the energy estimate \eqref{eq:ineq}. 

We start with the $\alpha$-entropy estimate. 
As $\{ u_{\eps,\delta}\}_{\varepsilon, \delta > 0}$ is uniformly bounded in 
$L^\infty(0,T;H^{s}(\Omega))$ and $\{\partial_t u_{\eps,\delta}\}_{\varepsilon, \delta > 0}$  is uniformly bounded in $L^2(0,T;\left(W^{1,q}(\Omega)\right)^*)$ with $q\geqslant 2$ (see Proposition \ref{prop:estimates_eps_delta}) then, by Aubin-Lions compactness lemma, we have 
\begin{equation}\label{eq:rrr-1}
u_{\eps,\delta} \xrightarrow{(\eps,\delta)\to (0,0)} u  \qquad \text{ strongly in } L^{2}(\Omega_T)  \textrm{ and a.e. in } \Omega_T. 
\end{equation} 
Since  $z^{\frac{\alpha +2}{2}}$ is continuous in $[0, +\infty)$ for $\frac{\alpha +2}{2}  > 0$ then 
\begin{equation}\label{eq:rrr-2}
u_{\eps,\delta}^{\frac{\alpha +2}{2}} \xrightarrow{(\eps,\delta)\to (0,0)} u^{\frac{\alpha +2}{2}}  \qquad   \textrm{ a.e. in } \Omega_T. 
\end{equation} 
Note that by \eqref{eq:alpha_entropy_2_globalt} we obtain $\{ u^{\frac{\alpha +2}{2}}_{\eps,\delta}\}_{\varepsilon, \delta > 0}$ is uniformly bounded in 
$L^2(0,T;H^{1+s}(\Omega))$, and by (\ref{eq:energy_esti2}) we have $\{ u^{\frac{\alpha +2}{2}}_{\eps,\delta}\}_{\varepsilon, \delta > 0}$ is uniformly bounded in 
$L^\infty(0,T;L^{l}(\Omega))$ with $2 < l \leqslant \frac{4d}{(d-2s)_+(\alpha +2)}$, i.\,e. $\alpha < \frac{4s}{(d-2s)_+}$.
From here, we obtain that  $\{ u^{\frac{\alpha +2}{2}}_{\eps,\delta}\}_{\varepsilon, \delta > 0}$ is uniformly bounded in 
$L^m( \Omega_T)$ for some $2 < m \leqslant l + 2 - \frac{ l(d-2(s+1))_+}{d}$. Using this boundedness, we arrive at 
\begin{equation}\label{eq:rrr-3}
 \{ u^{\frac{\alpha +2}{2}}_{\eps,\delta}\}_{\varepsilon, \delta > 0} \text{ is uniformly  integrable in }
 L^2( \Omega_T) .  
\end{equation}
So, due to (\ref{eq:rrr-2}) and (\ref{eq:rrr-3}), we can apply Vitali's convergence theorem and obtain that
\begin{equation}\label{eq:rrr-4}
u_{\eps,\delta}^{\frac{\alpha +2}{2}} \xrightarrow{(\eps,\delta)\to (0,0)} u^{\frac{\alpha +2}{2}}  \qquad   \text{ strongly in } L^{2}(\Omega_T). 
\end{equation} 
Consequently,
\[
\lim_{(\eps,\delta)\to (0,0)}\iint_{\Omega_T}u_{\eps,\delta}^{\alpha+2}(x,t)\d x \d t= \iint_{\Omega_T}u^{\alpha+2}(x,t)\d x\d t.
\]
The convergence above is sufficient to control the right-hand side in the $\alpha$-entropy estimate. 
However, in the running assumption, we have more, namely a strong convergence in $L^2(0,T;H^1(\Omega))$
for $u_{\eps,\delta}^{\frac{\alpha+2}{2}}$.
In fact, we recall that (see \cite[Lemma 9, p. 85]{simon87})
%
%
\begin{lemma}\label{simon-lem9}
Let $X \subset B \subset Y$ be Banach spaces with compact embedding $X \hookrightarrow B$ and let $p\in[1,+\infty] $. 
Let $\mathfrak{F}$ be bounded a subset of $L^p(0,T;Y)$ that is bounded in $ L^p(0,T; X) $ and relatively compact in $L^p(0,T; Y)$. 
Then $ \mathfrak{F} $ is relatively compact in $L^p(0,T; B)$.
\end{lemma}
Therefore, applying  Lemma~\ref{simon-lem9} with the choices $\mathfrak{F}= \bigl\{u_{\eps,\delta}^{\frac{\alpha+2}{2}}\bigr\}_{\varepsilon, \delta > 0}$, $p = 2$, $X= H^{s+1}(\Omega)$, $B = H^{1}(\Omega)$ and $Y= L^2(\Omega)$, we get
\begin{equation}\label{eq:rrr-5}
u_{\eps,\delta}^{\frac{\alpha +2}{2}} \xrightarrow{(\eps,\delta)\to (0,0)} u^{\frac{\alpha +2}{2}}  \qquad  
\text{ strongly in } L^{2}(0,T; H^1(\Omega)). 
\end{equation}

We continue with the limit procedure in the $\alpha$-entropy estimate dealing with the left-hand side. 
Now recall that (see \eqref{eq:alpha-entropy-approx})
\[
 \mathscr{G}^{\eps,\delta}_{\alpha}(z) = \mathscr{G}_{\alpha}(z) + \tfrac{\eps z^{\alpha - \beta + 2}}{(\alpha - \beta +2)(\alpha -\beta +1)} + \tfrac{\delta z^{\alpha+2}}{(\alpha+2)(\alpha +1)}, \text{ where } \beta > \alpha +2, \ \alpha > -1,
\]
Thus,  observe that $(\alpha - \beta +2)(\alpha -\beta +1)>0$, for almost all $t\in (0,T)$
\[
\mathscr{G}^{\eps,\delta}_{\alpha}(u_{\eps,\delta}(x,t))\geqslant \mathscr{G}_{\alpha}(u_{\eps,\delta}(x,t))\qquad \textrm{ a.e. in }\Omega. 
\] 
Therefore,  Fatou's Lemma implies
\[
\liminf_{(\eps,\delta)\to (0,0)}\int_{\Omega}\mathscr{G}^{\eps,\delta}_{\alpha}(u_{\eps,\delta}(x,t))\d x\geqslant \int_{\Omega}\mathscr{G}_{\alpha}(u_{\eps,\delta}(x,t))\d x.
\]
Using \eqref{H1p}, \eqref{Hsc-0} and \eqref{Hsc} and reasoning as in \cite[Section 4.4]{DNLST24}, we obtain 
\[
\lim_{(\eps,\delta)\to (0,0)}\int_{\Omega}\mathscr{G}^{\eps,\delta}_{\alpha}(u_{0,\eps,\delta})\d x = \int_{\Omega}\mathscr{G}_{\alpha}(u_{0})\d x.
\] 
Thus, thanks to the semicontinuity of norms w.r.t. weak convergence, we deduce that $u$ verifies 
\begin{equation}
\label{eq:alpha-entropy-u}
\int_{\Omega } \mathscr{G}_{\alpha}(u)   \d x  + \tfrac{4}{ (\alpha +2)^2}  \iint_{\Omega_T} \abs{\Dusm (u^{\frac{\alpha +2}{2}} ) }^2 \d x \d t \leqslant
\int_{\Omega } \mathscr{G}_{\alpha}(u_{0}) \d x  + C_1 \iint_{\Omega_T} { u^{\alpha +2 }  \d x \d t } .
\end{equation}

We now discuss the validity of energy estimate \eqref{eq:ineq}. To this end note that, as in \cite[Subsection 4.4.3]{DNLST24} $u_{\eps,\delta}$ verifies 
\begin{equation}
\label{eq:energy_esti2a}
	\norm{u_{\eps,\delta}(\cdot,t)}^{2}_{\dot{H}^s(\Omega)}
	+ 2 \iint_{\Omega_t}  \abs{g_{\varepsilon, \delta}}^2 \,\d x \,\d \tau
	\leqslant   \norm{u_{0,\eps,\delta}}^2_{\dot{H}^s(\Omega)} \qquad\text{for all } \,t\in (0,T],
\end{equation}
where the vector field $g_{\eps,\delta}$ is identified as 
\[
g_{\eps,\delta}=
\begin{cases}
m_{\eps,\delta}^{1/2}(u_{\eps,\delta})\nabla p_{\eps,\delta} &\qquad \textrm{ on } \,\left\{(x,t) \in \Omega_T: \, u_{\eps,\delta}(x,t) >0\right\},\\
0&\qquad \textrm{ on } \left\{(x,t) \in \Omega_T:\, u_{\eps,\delta}(x,t)=0\right\}.
\end{cases}
\]
Therefore, \eqref{eq:ineq} easily follows by semicontinuity arguments. The justification of (\ref{eq:pseudo_flux_weak}) is as in \cite[Subsection 4.4.3]{DNLST24}.

Next, we  prove the identification (\ref{eq:pseudo_flux}) for the pseudo-flux $g$, under the hypothesis that $n > \frac{4(s+1)}{4(s+1) - d}$.
We observe that $m^{1/2}_{\eps,\delta}$ is Lipschitz continuous and that $m'_{\eps,\delta }(0)=0$. 
Let us fix $\bm{\phi} \in C_c^\infty ({\Omega}_T;\R^d)$. 
We can write
$$ 
\iint_{\Omega_T} \nabla(m^{1/2}_{\eps,\delta}(u_{\eps,\delta}))\cdot \bm{\phi} \,\d x \,\d t = 
 \tfrac{1}{2}\iint_{\Omega_T}m^{-1/2}_{\eps,\delta}(u_{\eps,\delta})m'_{\eps,\delta}(u_{\eps,\delta})\nabla u_{\eps,\delta}\cdot \bm{\phi} \,\d x \,\d t .$$
Then we write
\begin{align*} 
\iint_{\Omega_T}m^{1/2}_{\eps,\delta}(u_{\eps,\delta})\nabla p_{\eps,\delta}\cdot\bm{\phi} \,\d x \,\d t &=
-\iint_{\Omega_T} {m^{1/2}_{\eps,\delta}(u_{\eps,\delta})p_{\eps,\delta} \div\bm{\phi} \,\d x \, \d t}\\ 
&\quad - \tfrac{1}{2}\iint_{\Omega_T} { m^{-1/2}_{\eps,\delta}(u_{\eps,\delta})m'_{\eps,\delta}(u_{\eps,\delta})p_{\eps,\delta} \nabla u_{\eps,\delta} \cdot \bm{\phi} \,\d x \,\d t}\\
&  = -\iint_{\Omega_T} {m^{1/2}_{\eps,\delta}(u_{\eps,\delta})p_{\eps,\delta} \div\bm{\phi} \,\d x \, \d t}\\
& \quad - \tfrac{1}{\alpha +2}\iint_{\Omega} u_{\eps,\delta}^{\frac{n}{2}-1-\frac{\alpha}{2} }\tfrac{m'_{\eps,\delta}(u_{\eps,\delta})}{( m_{\eps,\delta}(u_{\eps,\delta}) u_{\eps,\delta}^{n-2})^{1/2}}
p_{\eps,\delta}\nabla \left(u_{\eps,\delta}^{\frac{\alpha+2}{2}}\right)\cdot \bm{\phi} \, \d x \d t\\
& = I_1(\eps,\delta) + I_2(\eps,\delta) .
\end{align*}
Limit processes in $I_1(\eps,\delta)$ and $I_2(\eps,\delta)$ are similar to previous ones. However, we need 
to control $  m^{1/2}_{\eps,\delta}(u_{\eps,\delta}) $ and $\nabla (m^{1/2}_{\eps,\delta}(u_{\eps,\delta}))$
instead of $  m_{\eps,\delta}(u_{\eps,\delta}) $ and $\nabla (m_{\eps,\delta}(u_{\eps,\delta}))$. 
Due to \eqref{eq:m'}, \eqref{eq:entro1} and  \eqref{eq:conv_eps_delta}, we get that there exists 
a constant $C$ independent of $\eps,\delta$ such that 
\[
\begin{split}
\int_{0}^T \| \nabla (m^{1/2}_{\eps,\delta}(u_{\eps,\delta}))\|^2_{L^{r}(\Omega)}  \d t &\leqslant
C \int_{0}^T \|  u_{\eps,\delta}^{\frac{n}{2}-1-\frac{\alpha}{2} } \nabla (u_{\eps,\delta}^{\frac{\alpha+2}{2}})\|^2_{L^{r}(\Omega)}  \d t
  \\
&  \leqslant C \int_0^T \| \nabla (u_{\eps,\delta}^{\frac{\alpha+2}{2}}) \|^2_{L^{r_1}(\Omega)} \| u_{\eps,\delta}^{\frac{n}{2}-1-\frac{\alpha}{2}} \|^2_{L^{r_2}(\Omega)} \d t  \leqslant C 
\end{split}
\]
provided $
0 < r_2 (\tfrac{n}{2}-1-\tfrac{\alpha}{2}) \leqslant p_s. 
$
This is true if $2 + \alpha < n < \frac{2(d+2(1-s))}{d -2s} + \alpha$. On the other hand, by (\ref{rrt-3}) we have
$$
\int_{0}^T \|   m^{1/2}_{\eps,\delta}(u_{\eps,\delta})   \|^2_{L^{r}(\Omega)}  \d t \leqslant
\int_{0}^T \|  u_{\eps,\delta}  \|^{n}_{L^{ \frac{nr}{2}}(\Omega)}  \d t
\leqslant    C 
$$
provided
$ 
p = \tfrac{n r}{2} \leqslant \tfrac{p_sq}{q - \alpha - 2}, \  q = n > \alpha + 2. 
$ 
This is true if $2 + \alpha < n < \frac{2(d+2(1-s))}{d -2s} + 2 + \alpha$. We thus obtain the following final lower bound for $n$:
$$
n > 1 + \tfrac{d}{4(s+1) - d} = \tfrac{4(s+1)}{4(s+1) - d}.
$$

\section{Finite Speed of Propagation} \label{sec:fs}

\subsection{Local $\alpha$-entropy estimate}
Now, we prove a local version of Proposition \ref{prop:alpha_entropy_est-new}.
\begin{lemma}[local $\alpha$-entropy estimate]\label{lm:localestimate-000}
\label{lemma:local_entropy}
Let $s\in (\frac{(d-2)_+}{2},1)$ and $n > 0$. Given $u_{0,\eps,\delta}$ satisfying \eqref{eq:hyp_u0_ex}, we let
$u=u_{\eps,\delta} > 0$ be a solution of \eqref{eq:approx_old} with $\gamma = 0$.
Then, for every  $\psi \in C^{1,2}_{t,x} (\bar{\Omega}_T)$ such that $\nabla \psi \cdot  {\bf{n}}  = 0$
 on $\partial \Omega \times [0, \infty)$, the following holds
\begin{equation}
\label{eq:weighted_entropy1}
\begin{split}
&  \int_{\Omega} \mathscr{G}_{\alpha}^{\eps,\delta}(u) \psi^3 \d x - \iint_{\Omega_T} \mathscr{G}_{\alpha}^{\eps,\delta}(u) (\psi^3)_t \d x \d t 
 + \tfrac{4}{ (\alpha +2)^2} \iint_{\Omega_T}  \abs{\Dusm (u^{\frac{\alpha +2}{2}} \psi) }^2 \psi \, \d x \d t \\
  &=  \int_{\Omega} \mathscr{G}^{\eps,\delta}_\alpha(u_{0,\eps,\delta}) \psi^3(x,0) \d x    
  - \tfrac{4}{ (\alpha +2)^2} \iint_{\Omega_T}\Dusm (u^{\frac{\alpha +2}{2}} \psi) \, u^{\frac{\alpha + 2}{2}} \psi \Dusm \psi \, \d x \d t \\
 &+ \tfrac{2}{(\alpha +1)(\alpha +2)} \iint_{\Omega_T} u^{\alpha +2} \psi^2 (- \Delta)^{s+1} \psi \, \d x \d t  +  
 \tfrac{ \alpha  }{(\alpha +1) ( \alpha +2)^2 } \iint \limits_{\Omega_T} { u^{\alpha +2} (-\Delta)^{s+1} \psi^3  \, \d x \d t   } \\
 &- \tfrac{4}{ (\alpha +2)^2} \iint_{\Omega_T}   \Dusm  (u^{\frac{\alpha + 2}{2}} \psi) \, R_{\frac{1+s}{2}}(u^{\frac{\alpha + 2}{2}} \psi,\psi) \, \d x \d t \\
& + \tfrac{4}{ (\alpha +2)^2} \iint_{\Omega_T}  u^{\frac{\alpha + 2}{2}} \psi^2 R_{s+1} ( u^{\frac{\alpha+2}{2}}, \psi) \, \d x \d t  \\
& -  \tfrac{  \alpha }{(\alpha +1)( \alpha +2)^2 \Gamma(-s-1)}    \iint_{\Omega_T} {  \psi^3 \int_0^{+\infty} { \Bigl( \int_{\Omega} {  K(x,y,\tau) F(u(x), u(y)) \,dy}   \Bigr) \tfrac{d\tau}{\tau^{s+2}}} \d x \d t}  \\
  &+ \tfrac{6 }{\alpha +2} \iint_{\Omega_T}  p u^{\frac{\alpha}{2}} \psi \nabla (u^{\frac{\alpha+2}{2}} \psi) \cdot\nabla \psi \, \tfrac{ m'(u) \mathscr{G}'_{\alpha}(u) }{u^{\alpha}} \, \d x \d t \\
 & + 6 \iint_{\Omega_T}  p u^{\frac{\alpha}{2}} \psi u^{\frac{\alpha+2}{2}} \bigl[  \tfrac{m(u) \mathscr{G}'_{\alpha}(u)}{u^{\alpha+1}} - \tfrac{1}{\alpha+1} - \tfrac{1}{\alpha +2} \tfrac{ m'(u) \mathscr{G}'_{\alpha}(u) }{u^{\alpha}} \bigr]  | \nabla \psi|^2 \, \d x \d t \\
  &+ 3 \iint_{\Omega_T}  p u^{\frac{\alpha}{2}} \psi u^{\frac{\alpha+2}{2}} \bigl[  \tfrac{m(u) \mathscr{G}'_{\alpha}(u)}{u^{\alpha+1}} - \tfrac{1}{\alpha+1}   \bigr]  \psi \Delta \psi \, \d x \d t \\
  & = \int_{\Omega} \mathscr{G}^{\eps,\delta}_\alpha(u_{0,\eps,\delta}) \psi^3(x,0) \d x  + 
   \sum_{i=1}^{9}I_{i},
\end{split}
\end{equation}	
where the function $F$ is defined in (\ref{rrt-fff}).
\end{lemma}

\begin{proof}[Proof of Lemma~\ref{lm:localestimate-000}]

 Multiplying (\ref{eq:ft}) by $\mathscr{G}'_{\alpha}(u) \psi^3$, we have
\begin{equation}
\label{eq:conto1}
\begin{split}
& \frac{\d}{\d t}\int_\Omega \mathscr{G}_{\alpha}(u) \psi^3 \d x  -   \int_\Omega \mathscr{G}_{\alpha}(u) (\psi^3)_t \d x  = \int \psi^3 \mathscr{G}_{\alpha}'(u) \div(m(u)\nabla p) \d x \\
& =   -\int_\Omega \psi^3 m(u) \mathscr{G}''_{\alpha}(u)\nabla u\cdot \nabla p\d x - \int_\Omega m(u) \mathscr{G}'_{\alpha}(u) \nabla \psi^3 \cdot \nabla p\d x \\
 & =  -\int_\Omega \psi^3 u^{\alpha} \nabla u\cdot \nabla p\d x  -    \int_\Omega m(u) \mathscr{G}'_{\alpha}(u) \nabla \psi^3 \cdot \nabla p \d x \\
& =  \tfrac{1}{\alpha + 1}\int_\Omega u^{\alpha +1} \psi^3 \Delta p\d x  + \tfrac{1}{\alpha+1}\int_\Omega \nabla (\psi^3)\cdot \nabla p u^{\alpha+1}\d x \\
&-    \int_\Omega m(u) \mathscr{G}'_{\alpha}(u) \nabla \psi^3 \cdot \nabla p \d x \\
& =  \tfrac{1}{\alpha + 1}\int_\Omega u^{\alpha +1} \psi^3 \Delta p\d x -  \int_\Omega ( m(u) \mathscr{G}'_{\alpha}(u) -
\tfrac{1}{\alpha+1} u^{\alpha +1} ) \nabla \psi^3 \cdot \nabla p \d x \\
&=  \tfrac{1}{\alpha + 1} A + B  .
\end{split}
\end{equation}

Multiplying (\ref{rrt-ktr}) by $\psi$ and using the Fractional Leibniz formula as follows
$$
(-\Delta)^{s+1}\left(u^{ \frac{\alpha+2}{2}}\psi\right) =  
  \psi (-\Delta)^{s+1} u^{\frac{\alpha+2}{2}} + u^{ \frac{\alpha+2}{2}} (-\Delta )^{s+1} \psi  + R_{s+1}\left(u^{\frac{\alpha+2}{2}},\psi\right),
$$
we deduce that
$$
(- \Delta)^{s+1} ( u^{\frac{\alpha +2}{2}} \psi) = \tfrac{\alpha +2}{2} u^{\frac{\alpha}{2}} \psi (- \Delta)^{s+1} u -
\psi \mathcal{J}_{s+1} [u]  + u^{ \frac{\alpha+2}{2}} (-\Delta )^{s+1} \psi  + R_{s+1}\left(u^{\frac{\alpha+2}{2}},\psi\right).
$$
Multiplying this equality by $u^{\frac{\alpha +2}{2}} \psi^2$, using
$$
(-\Delta)^{\frac{s+1}{2}}\left(u^{ \frac{\alpha+2}{2}}\psi^2\right) =  
  \psi (-\Delta)^{\frac{s+1}{2}} ( u^{\frac{\alpha+2}{2}} \psi) + u^{ \frac{\alpha+2}{2}} \psi (-\Delta )^{\frac{s+1}{2}} \psi  + R_{\frac{s+1}{2}}\left(u^{\frac{\alpha+2}{2}}\psi,\psi\right),
$$
and after then integrating over $\Omega$, we have
\[
\begin{split}
\int_\Omega  \bigl| (- \Delta)^{\frac{s+1}{2}} (u^{\frac{\alpha +2}{2}} \psi) \bigr|^2 \psi \,  \d x &=  
\tfrac{\alpha +2}{2} \int_\Omega u^{\alpha +1} \psi^3 (- \Delta)^{s+1} u \, \d x -
\int_\Omega u^{\frac{\alpha + 2}{2}} \psi^3 \mathcal{J}_{s+1} [u] \, \d x  \\
&+ \int_\Omega u^{\alpha +2} \psi^2 (- \Delta)^{s+1} \psi \, \d x  - \int_\Omega   (- \Delta)^{\frac{s+1}{2}} (u^{\frac{\alpha +2}{2}} \psi) u^{\frac{\alpha + 2}{2}} \psi (- \Delta)^{\frac{s+1}{2}} \psi  \,  \d x  \\
& +\int_\Omega u^{\frac{\alpha + 2}{2}} \psi^2 R_{s+1}\left(u^{\frac{\alpha+2}{2}},\psi\right) \, \d x - 
\int_\Omega   (- \Delta)^{\frac{s+1}{2}} (u^{\frac{\alpha +2}{2}} \psi) R_{\frac{s+1}{2}}\left(u^{\frac{\alpha+2}{2}}\psi,\psi\right)  \,  \d x . 
\end{split}
\]
%
%
%
Therefore,
\begin{equation} \label{eq:A}
\begin{split}
A &=   - \int_\Omega u^{\alpha +1} \psi^3  ( -\Delta)^{s+1} u \d x  \\
  & = - \tfrac{2}{\alpha +2} \int_\Omega  \bigl| (- \Delta)^{\frac{s+1}{2}} (u^{\frac{\alpha +2}{2}} \psi) \bigr|^2 \psi \,  \d x - \tfrac{2}{\alpha +2} \int_\Omega u^{\frac{\alpha + 2}{2}} \psi^3 \mathcal{J}_{s+1} [u] \, \d x \\
& + \tfrac{2}{\alpha +2} \int_\Omega u^{\alpha +2} \psi^2 (- \Delta)^{s+1} \psi \, \d x - \tfrac{2}{\alpha +2} \int_\Omega   (- \Delta)^{\frac{s+1}{2}} (u^{\frac{\alpha +2}{2}} \psi) u^{\frac{\alpha + 2}{2}} \psi (- \Delta)^{\frac{s+1}{2}} \psi  \,  \d x \\
 & + \tfrac{2}{\alpha +2} \int_\Omega u^{\frac{\alpha + 2}{2}} \psi^2 R_{s+1}\left(u^{\frac{\alpha+2}{2}},\psi\right) \, \d x- \tfrac{2}{\alpha +2} \int_\Omega   (- \Delta)^{\frac{s+1}{2}} (u^{\frac{\alpha +2}{2}} \psi) R_{\frac{s+1}{2}}\left(u^{\frac{\alpha+2}{2}}\psi,\psi\right)  \,  \d x .
\end{split}
\end{equation}
%
%
Using a local version of (\ref{ps-new}), i.\,e. 
 \begin{equation}
 \begin{split}
 \label{ps-new-local}
&\int_{\Omega} {  |(-\Delta)^{\frac{\mu}{2}} (v \psi) |^2 \psi \, dx   }  \\
 &= \tfrac{1}{2} \int_{\Omega} { v^2 (-\Delta)^{ \mu } \psi^3(x)  \, dx   }
 -   \tfrac{1}{ \Gamma(1-\mu)}
\int_{\Omega}{ \psi^3(x) |\nabla v(x)|^2 \int_0^{+\infty} { \Bigl( \int_{\Omega} {  K(x,y,t) dy}   \Bigr) \tfrac{dt}{t^{\mu}}} dx}  \\
&-\tfrac{1}{2 \Gamma(-\mu)} \int \limits_{\Omega}{\psi^3(x) \int \limits_0^{+\infty} { \Bigl(
\int_{\Omega} { K(x,y,t) (v(y) - v(x))^2 dy}   \Bigr) \tfrac{dt}{t^{1+\mu}}} dx}  \\
&-\int_{\Omega} { v \psi (-\Delta)^{\frac{\mu}{2}} \psi  \, (-\Delta)^{\frac{\mu}{2}} (v \psi) \, dx   } - 
\int_{\Omega} { (-\Delta)^{\frac{\mu}{2}} (v \psi) R_{\frac{\mu}{2}} (v\psi,\psi) \, dx   }  \\
&+\int_{\Omega} { v \psi^2 R_{ \mu } (v ,\psi) \, dx   } ,
\end{split}
 \end{equation}
with $\mu = s+1$ and $v = u^{\frac{\alpha + 2}{2}}$,  we have
\[
\begin{split}
\int_\Omega u^{\frac{\alpha + 2}{2}}\psi^3 \mathcal{J}_{s+1} [u] \, \d x &=
- \tfrac{ \alpha  }{  \alpha +2 } \tfrac{1}{ \Gamma( -s)}  \int_{\Omega} {\psi^3(x) |\nabla u^{\frac{\alpha +2}{2}}(x)|^2  \int_0^{+\infty} { \Bigl( \int_{\Omega} {  K(x,y,t) dy}   \Bigr) \tfrac{dt}{t^{s+1}}} dx}  \\
&+\int_\Omega u^{\frac{\alpha + 2}{2}}\psi^3 \mathcal{I}_{s+1} [u] \, \d x =
\tfrac{\alpha}{ \alpha +2 } \int_\Omega  \bigl| (- \Delta)^{\frac{s+1}{2}} (u^{\frac{\alpha +2}{2}} \psi) \bigr|^2 \psi \,  \d x  \\
&+\tfrac{\alpha}{2( \alpha +2)\Gamma(-s-1)}  \int_{\Omega} { \psi^3  \int_0^{+\infty} { \Bigl( \int_{\Omega} {  K(x,y,t) (u^{\frac{\alpha + 2}{2}}(y) - u^{\frac{\alpha + 2}{2}}(x))^2 \,dy}   \Bigr) \tfrac{dt}{t^{s+2}}} dx}  \\
&-\tfrac{\alpha(\alpha +2)}{4\Gamma(-s-1)} \int_\Omega { u^{\frac{\alpha + 2}{2}}(x)\psi^3 
\int_0^{+\infty} { \Bigl( \int_{\Omega} {  K(x,y,t) \int_{u(x)}^{u(y)} {  z^{\frac{\alpha -2}{2}} (u(y) - z) dz} \,dy}   \Bigr) \tfrac{dt}{t^{s+2}}} \, \d x }  \\
&-  \tfrac{ \alpha  }{ 2( \alpha +2) } \int_{\Omega} { u^{\alpha +2} (-\Delta)^{s+1} \psi^3   \, dx   } + 
\tfrac{ \alpha  }{  \alpha +2 } \int_{\Omega} { u^{\frac{\alpha + 2}{2}} \psi (-\Delta)^{\frac{s+1}{2}} \psi  \, (-\Delta)^{\frac{s+1}{2}} (u^{\frac{\alpha + 2}{2}} \psi) \, dx   }  \\
&+\tfrac{ \alpha  }{  \alpha +2 } \int_{\Omega} { (-\Delta)^{\frac{s+1}{2}} (u^{\frac{\alpha + 2}{2}} \psi) R_{\frac{s+1}{2}} (u^{\frac{\alpha + 2}{2}}\psi,\psi) \, dx   } -  
\tfrac{ \alpha  }{  \alpha +2 } \int_{\Omega} { u^{\frac{\alpha + 2}{2}} \psi^2 R_{ s+1 } (u^{\frac{\alpha + 2}{2}} ,\psi) \, dx   }.
\end{split}
\]
Using this equality in (\ref{eq:A}), we deduce that
\begin{equation} \label{eq:A-2}
\begin{split}
A &=  - \tfrac{4(\alpha+1)}{(\alpha +2)^2} \int_\Omega  \bigl| (- \Delta)^{\frac{s+1}{2}} (u^{\frac{\alpha +2}{2}} \psi) \bigr|^2 \psi \,  \d x   \\
 & + \tfrac{2}{\alpha +2} \int_\Omega u^{\alpha +2} \psi^2 (- \Delta)^{s+1} \psi \, \d x +  
 \tfrac{ \alpha  }{ ( \alpha +2)^2 } \int_{\Omega} { u^{\alpha +2} (-\Delta)^{s+1} \psi^3  \, dx   } \\
 &- \tfrac{4(\alpha+1)}{(\alpha +2)^2} \int_\Omega   (- \Delta)^{\frac{s+1}{2}} (u^{\frac{\alpha +2}{2}} \psi) u^{\frac{\alpha + 2}{2}} \psi (- \Delta)^{\frac{s+1}{2}} \psi  \,  \d x \\
 & + \tfrac{4(\alpha+1)}{(\alpha +2)^2}\int_\Omega u^{\frac{\alpha + 2}{2}} \psi^2 R_{s+1}\left(u^{\frac{\alpha+2}{2}},\psi\right) \, \d x - \tfrac{4(\alpha+1)}{(\alpha +2)^2} \int_\Omega   (- \Delta)^{\frac{s+1}{2}} (u^{\frac{\alpha +2}{2}} \psi) R_{\frac{s+1}{2}}\left(u^{\frac{\alpha+2}{2}}\psi,\psi\right)  \,  \d x \\
 & - \tfrac{\alpha}{ ( \alpha +2)^2 \Gamma(-s-1)}  \int_{\Omega} { \psi^3  \int_0^{+\infty} { \Bigl( \int_{\Omega} {  K(x,y,t) (u^{\frac{\alpha + 2}{2}}(y) - u^{\frac{\alpha + 2}{2}}(x))^2 \,dy}   \Bigr) \tfrac{dt}{t^{s+2}}} dx} \\
 & + \tfrac{\alpha }{2 \Gamma(-s-1)} \int_\Omega { u^{\frac{\alpha + 2}{2}}(x)\psi^3 
\int_0^{+\infty} { \Bigl( \int_{\Omega} {  K(x,y,t) \int_{u(x)}^{u(y)} {  z^{\frac{\alpha -2}{2}} (u(y) - z) dz} \,dy}   \Bigr) \tfrac{dt}{t^{s+2}}} \, \d x }.
\end{split}
\end{equation}

On the other hand,
\begin{equation}\label{eq:A-0}
\begin{split}
  B = & \int_\Omega p\div \bigl(( m(u) \mathscr{G}'_{\alpha}(u) -
\tfrac{1}{\alpha+1} u^{\alpha +1}) \nabla \psi^3\bigr) \d x\\
    & = \int_\Omega p \, m'(u) \mathscr{G}'_{\alpha}(u) \nabla u \cdot \nabla \psi^3 \d x + \int_\Omega p \, ( m(u) \mathscr{G}'_{\alpha}(u) - \tfrac{1}{\alpha+1} u^{\alpha +1}) \Delta \psi^3 \d x \\
     &  =    \tfrac{6 }{\alpha +2} \int_\Omega p u^{\frac{\alpha}{2}} \psi \nabla (u^{\frac{\alpha+2}{2}} \psi) \cdot\nabla \psi \, \tfrac{ m'(u) \mathscr{G}'_{\alpha}(u) }{u^{\alpha}} \d x  -  \tfrac{6 }{\alpha +2} \int_\Omega p(u) u^{\frac{\alpha}{2}} \psi u^{\frac{\alpha+2}{2}} \tfrac{ m'(u) \mathscr{G}'_{\alpha}(u) }{u^{\alpha}}  | \nabla \psi|^2  \d x \\
     & + 6 \int _\Omega p \, ( m(u) \mathscr{G}'_{\alpha}(u) - \tfrac{1}{\alpha+1} u^{\alpha +1}) \psi |\nabla \psi|^2 \d x
       +      3 \int_\Omega p \, ( m(u) \mathscr{G}'_{\alpha}(u) - \tfrac{1}{\alpha+1} u^{\alpha +1}) \psi^2 \Delta \psi  \d x  \\
      &  =    \tfrac{6 }{\alpha +2} \int_\Omega p u^{\frac{\alpha}{2}} \psi \nabla (u^{\frac{\alpha+2}{2}} \psi) \cdot\nabla \psi \, \tfrac{ m'(u) \mathscr{G}'_{\alpha}(u) }{u^{\alpha}} \d x \\
      & + 6 \int_\Omega p u^{\frac{\alpha}{2}} \psi u^{\frac{\alpha+2}{2}} \bigl[  \tfrac{m(u) \mathscr{G}'_{\alpha}(u)}{u^{\alpha+1}} - \tfrac{1}{\alpha+1} 
      - \tfrac{1}{\alpha +2} \tfrac{ m'(u) \mathscr{G}'_{\alpha}(u) }{u^{\alpha}} \bigr]  | \nabla \psi|^2  \d x \\
        & + 3 \int_\Omega p u^{\frac{\alpha}{2}} \psi u^{\frac{\alpha+2}{2}} \bigl[  \tfrac{m(u) \mathscr{G}'_{\alpha}(u)}{u^{\alpha+1}} - \tfrac{1}{\alpha+1}   \bigr]  \psi \Delta \psi  \d x .
        \end{split}
\end{equation}
Time integration leads to the relation (\ref{eq:weighted_entropy1}).
\end{proof}


For an arbitrary $S > 0$  we consider the sets
\begin{equation}\label{e-7}
\Omega (S) := \Omega \setminus \{ x \in \bar{\Omega} : |x | \leqslant  S \}, \qquad \Omega_T(S) := (0,T)  \times \Omega(S).
\end{equation}
Given $\delta > 0$, we consider the cut-off functions $\psi_{S,\delta }
\in C^{\infty} (\mathbb{R})$ such that $0 \leqslant \psi_{S,\delta} \leqslant 1$, and 
\begin{equation}\label{e-9}
\psi_{S,\delta } (x) =
\begin{cases}
1 & \text{ if } x \in \Omega(S +  \delta),\\
0 & \text{ if } x \in  \Omega \setminus \Omega(S),
\end{cases}
\end{equation}
$$
|  \nabla \psi_{S,\delta }| \leqslant \tfrac{C} {  \delta },\quad
| (-\Delta)^{ s +1 } \psi_{S,\delta } | \leqslant \tfrac{C} {  \delta^{2(s+1)} }
$$
for all $ x \in \mathbb{R}$, $s \in [0,1)$.

Finally, we introduce the entropy concentrated on 
annuli
\begin{equation}
\label{eq:ene_annuli}
A_T(S) := \iint_ {\Omega_T(S)}{  u^{\alpha +2} \dd x \dd t}.	
\end{equation}
As we will see in a while, $A_T(S)$ will play a key role in the proof of the finite speed of propagation.
\begin{lemma}[Final local $\alpha$-entropy inequality]
\label{lm:add-000}
Let $s\in (\frac{(d-2)_+}{2},1)$ and $n\in (\tfrac{2(s+1)}{ 4(s+1)-d },2)$. Suppose the initial data $u_0$ satisfies \eqref{eq:hyp_u0_ex} and
$$
\int \limits_{\Omega} \mathscr{G}_\alpha(u_0)  \d x < \infty.
$$
Let  $u$ be a weak solution in the class $\mathscr{A}(u_0)$.
Then, for any $\delta > 0$, we have 
\begin{equation}
 \label{main-alp}
\begin{split}
 \int_{\Omega(S+  \delta)}{\mathscr{G}_{\alpha}(u) \,dx}  +
C \iint_{\Omega_T(S +\delta) }{ |(-\Delta)^{\frac{s+1}{2}}(u^{\frac{\alpha +2}{2}} \psi_{S,\delta } ) |^2   \,\d x  \d t}    \\
\leqslant  \int_{\Omega(S)}{\mathscr{G}_{\alpha}(u_0) \,dx} +  \tfrac{C }{\delta^{2(s+1)}} A_T(S)  + \tfrac{C }{\delta^{2(s+1)}}
A^{1 - \varpi}_T(0) A^{\varpi}_T(S)  ,
\end{split}
\end{equation}
where $(n-1)_+ < \alpha \leqslant 1$, $\varpi := \frac{s}{2s+1}$.
\end{lemma}

\begin{proof}[Proof of Lemma~\ref{lm:add-000}]
First, we recall the following interpolation inequality (whose proof is similar to that of \cite[Lemma A.6]{DNLST24}):
\[
\begin{split}
\| \psi u^{\frac{\alpha}{2}} (- \Delta)^{\beta} u \|_{L^2(\Omega)}
\leqslant \| \Dusm (u^{\frac{\alpha +2}{2}} \psi)  \|^{\theta}_{L^2(\Omega) }
\|  u^{\frac{\alpha +2}{2}} \psi   \|^{1 - \theta}_{L^2(\Omega) }  + \tfrac{C}{\delta^{2\beta}}\|u^{\frac{\alpha +2}{2}} \|_{L^2(\Omega) } ,
\end{split}
\]
where $\beta \in (0, \frac{1+s}{2})$, $\alpha \geqslant 0$, and $ \theta = \frac{2\beta}{1+s}$.
Next, we recall that the commutator $R_\beta$ satisfies the estimate
$$
\| R_{\beta}(f,g) \|_{L^2(\Omega)} \leqslant \| f \|_{L^2(\Omega)} \| (-\Delta)^{\beta} g \|_{L^{\infty}(\Omega)} 
$$
for  $\beta \in (0,2)$. The proof of this estimate follows from \cite[Theorem 1.2]{Dong19} in the case
$f, g \in \mathcal{S}(\mathbb{R}^d)$, and from \cite[Theorem 1.1]{VazquezHungSire} in the setting of bounded domains.

Next, we assume that $(n-1)_+ < \alpha \leqslant 1$, and we choose the test function $\psi = \psi_{S,\delta }(x)$ in (\ref{eq:weighted_entropy1}). 
The terms on the right-hand side of  \eqref{eq:weighted_entropy1} can then be estimated as follows:
\[
\begin{split}
I_1 &=  \iint_{\Omega_T} \Dusm (u^{\frac{\alpha +2}{2}} \psi_{S,\delta }) u^{\frac{\alpha + 2}{2}} \psi_{S,\delta } \Dusm  \psi_{S,\delta } \d x \d t   \\
&\quad\leqslant \tfrac{1}{2} \iint_{\Omega_T}  \abs{ \Dusm (u^{\frac{\alpha +2}{2}} \psi_{S,\delta })}^2   \psi_{S,\delta } \d x  \d t  + \tfrac{1}{2} 
 \iint_{\Omega_T}  u^{ \alpha + 2 } \psi_{S,\delta }  \abs{ \Dusm  \psi_{S,\delta } }^2 \d x  \d t  \\
&\quad\leqslant \tfrac{1}{2} \iint_{\Omega_T}  \abs{ \Dusm (u^{\frac{\alpha +2}{2}} \psi_{S,\delta }) }^2   \psi_{S,\delta } \d x  + \tfrac{C}{\delta^{2(s+1)}}
\int_{\Omega(S)}u^{ \alpha + 2 }  \d x  \d t  ,
\end{split}
\]
\[
\begin{split}
I_2 + I_3 &= \iint_{\Omega_T}   u^{\alpha +2} \bigl(  \psi_{S,\delta }^2 (- \Delta)^{s+1} \psi_{S,\delta } +  \tfrac{ \alpha  }{2(\alpha +2)} (-\Delta)^{s+1} \psi_{S,\delta }^3 \bigr) \, \d x \d t\\
 & \leqslant 
\tfrac{C}{\delta^{2(s+1)}} \iint_{\Omega_T(S)} u^{ \alpha + 2 }  \d x \d t ,
\end{split}
\]

\[
\begin{split}
I_4 & =   \iint_{\Omega_T} \Dusm  (u^{\frac{\alpha + 2}{2}} \psi_{S,\delta }) \,  R_{\frac{1+s}{2}}(u^{\frac{\alpha + 2}{2}} \psi_{S,\delta },\psi_{S,\delta }) \d x \d t     \\
&\leqslant \sigma \int_0^T \norm{ \Dusm  (u^{\frac{\alpha +2}{2}}\psi_{S,\delta }) }^2_{L^2(\Omega )} \d t +
C(\sigma) \int_0^T \norm{ R_{\frac{1+s}{2}}(u^{\frac{\alpha + 2}{2}} \psi_{S,\delta },\psi_{S,\delta })}^2_{L^2(\Omega )} \d t \\
&\leqslant \sigma \int_0^T  \norm{\Dusm  (u^{\frac{\alpha +2}{2}} \psi_{S,\delta }) }^2_{L^2(\Omega )} \d t + \tfrac{C(\sigma)}{\delta^{2(s+1)}}
\iint_{\Omega_T(S)} u^{ \alpha + 2 }  \d x \d t,
\end{split}
\]
$$
 I_5 = \iint_{\Omega_T} u^{\frac{\alpha + 2}{2}} \psi_{S,\delta }^2 R_{s+1} (u^{\frac{\alpha+2}{2}}, \psi_{S,\delta }) \d x \d t \leqslant     
 \tfrac{C}{\delta^{2(s+1)}} \left( \iint_{\Omega_T(S)} u^{ \alpha + 2 }  \d x \d t \right)^{\frac{1}{ 2}}
 \left( \iint_{\Omega_T } u^{ \alpha + 2 }  \d x \d t \right)^{\frac{1}{ 2}},
$$

$$
I_6 =\iint_{\Omega_T } {  \psi_{S,\delta }^3(x) \int_0^{+\infty} { \Bigl( \int_{\Omega} {  K(x,y,\tau) F(u(x), u(y)) \,dy}   \Bigr) \tfrac{d\tau}{\tau^{s+2}}} \d x \d t }
\leqslant  C \iint_{\Omega_T(S)} u^{ \alpha + 2 }  \d x \d t,
$$

\[
\begin{split}
I_7 &=   \iint_{\Omega_T} p(u) u^{\frac{\alpha}{2}} \psi_{S,\delta } \nabla (u^{\frac{\alpha+2}{2}} \psi_{S,\delta }) \cdot \nabla \psi_{S,\delta } 
  \, \tfrac{ m'(u) \mathscr{G}'_{\alpha}(u) }{u^{\alpha}}  \d x \d t
   \\
&\leqslant \tfrac{C}{\delta} \int_0^T \| p(u) u^{\frac{\alpha}{2}} \psi_{S,\delta }  \|_{L^2(\Omega )} \| \nabla (u^{\frac{\alpha+2}{2}} \psi_{S,\delta }) \|_{L^2(\Omega )} \d t  \\
&\leqslant \tfrac{C}{\delta} \int_0^T \| p(u) u^{\frac{\alpha}{2}} \psi_{S,\delta } \|_{L^2(\Omega )}
\| \Dusm  (u^{\frac{\alpha +2}{2}} \psi_{S,\delta }) \|^{\frac{1}{s+1}}_{L^2(\Omega )}\|  u^{\frac{\alpha +2}{2}} \psi_{S,\delta }  \|^{\frac{s}{s+1}}_{L^2(\Omega )} \d t
 \\
&\leqslant \sigma \int_0^T \| \Dusm  (u^{\frac{\alpha +2}{2}} \psi_{S,\delta }) \|^2_{L^2(\Omega )} \d t +
\tfrac{C(\sigma)}{\delta^{2(s+1)}} \iint_{\Omega_T(S)} u^{ \alpha + 2 }  \d x \d t \\
& + 
\tfrac{C(\sigma)}{\delta^{2(s+1)}} \|  u^{\frac{\alpha +2}{2}} \|^{\frac{2(1+s)}{2s+1}}_{L^2(\Omega_T )}
\|  u^{\frac{\alpha +2}{2}} \psi_{S,\delta }  \|^{\frac{2s}{2s+1}}_{L^2(\Omega_T )},
\end{split}
\]
\[
\begin{split}
I_8 + I_9 &= 3 \iint_{\Omega_T} p(u) u^{\frac{\alpha}{2}}\psi_{S,\delta } u^{\frac{\alpha+2}{2}} \left( 2\left[  \tfrac{m(u) \mathscr{G}'_{\alpha}(u)}{u^{\alpha+1}} 
- \tfrac{1}{\alpha+1} - \tfrac{1}{\alpha +2} \tfrac{ m'(u) \mathscr{G}'_{\alpha}(u) }{u^{\alpha}} \right]  \abs{\nabla \psi_{S,\delta }}^2 \right. \\
&+ \left.\left[  \tfrac{m(u) \mathscr{G}'_{\alpha}(u)}{u^{\alpha+1}} - \tfrac{1}{\alpha+1}   \right] \psi_{S,\delta } \Delta \psi_{S,\delta } \right)  
\d x \d t \\
&\leqslant  \tfrac{C}{\delta^2} \int_0^T \norm{ p(u) u^{\frac{\alpha}{2}} \psi_{S,\delta } }_{L^2(\Omega )}  \norm{ u^{\frac{\alpha +2}{2}} }_{L^2(\Omega(S) )} \d t \leqslant
  \sigma \int_0^T \norm{ \Dusm  (u^{\frac{\alpha +2}{2}} \psi_{S,\delta }) }^2_{L^2(\Omega )} \d t  \\
&+  \tfrac{C(\sigma)}{\delta^{2(s+1)}} \iint_{\Omega_T(S)} u^{ \alpha + 2 }  \d x \d t  +
\tfrac{C(\sigma)}{\delta^{2(s+1)}} \norm{  u^{\frac{\alpha +2}{2}}}_{L^2(\Omega_T )}
\norm{ u^{\frac{\alpha +2}{2}}}_{L^2(\Omega_T(S) )} .
\end{split}
\]

Since $u^{\frac{\alpha +2}{2}} \psi_{S,\delta } \in H^{s+1}_N (\Omega)$ and $\text{supp} (u^{\frac{\alpha +2}{2}} \psi_{S,\delta }) \subset \Omega(S) $ then,
due to (\ref{fr-1}) with $\mu = \frac{s+1}{2}$, we get
\begin{equation}\label{rrr-int}
\| (-\Delta)^{\frac{s+1}{2}} ( u^{\frac{\alpha +2}{2}} \psi_{S,\delta })  \|_{L^2(\Omega)} \leqslant   
 \| (-\Delta)^{\frac{s+1}{2}} (u^{\frac{\alpha +2}{2}} \psi_{S,\delta })  \|_{L^2(\Omega(S+\delta))} + \tfrac{C}{\delta^{s+1}}  \| u^{\frac{\alpha +2}{2}}   \|_{L^2(\Omega(S) )}.
\end{equation}
Using these estimates,  we eventually arrive at
\[
\begin{split}
&\int_{\Omega(S+  \delta)}{\mathscr{G}_{\alpha}^{\eps,\delta} (u) \,dx}  +
\tfrac{2}{(\alpha +2)^2} \iint_{\Omega_T(S +\delta) }{ |(-\Delta)^{\frac{s+1}{2}}(u^{\frac{\alpha +2}{2}} \psi_{S,\delta } ) |^2   \,\d x  \d t}   \\
&\leqslant \int_{\Omega(S)}{\mathscr{G}_{\alpha}^{\eps,\delta}(u_{0,\eps,\delta}) \,dx} +
\sigma \iint_{\Omega_T(S+\delta) }{ |(-\Delta)^{\frac{s+1}{2}}(u^{\frac{\alpha +2}{2}} \psi_{S,\delta } ) |^2   \,\d x  \d t}
+ \tfrac{C }{\delta^{2(s+1)}}   \iint_{\Omega_T(S)} {  u^{\alpha+ 2}  \, \d x \d t}  \\
&+\tfrac{C }{\delta^{2(s+1)}} \Bigl( \iint_{\Omega_T(S)} {  u^{\alpha+ 2}  \, dx dt} \Bigr)^{\frac{s}{2s+1}}
 \Bigl(\iint_{\Omega_T} {  u^{\alpha+ 2}  \, dx dt} \Bigr)^{\frac{s+1}{2s+1}}   \\
&+\tfrac{C}{\delta^{ 2(s+1) }}  \Bigl( \iint_{\Omega_T(S)} {  u^{\alpha+ 2}  \, \d x \d t} \Bigr)^{\frac{1}{2}}
 \Bigl(\iint_{\Omega_T} {  u^{\alpha+ 2}  \, dx dt} \Bigr)^{\frac{ 1}{2 }} ,
\end{split}
\]
whence choosing $\sigma\in (0,1)$ small enough and $T >0$, after letting $\varepsilon,\,\delta \to 0$, we obtain (\ref{main-alp}).
\end{proof}

\subsection{Proof of FSP for $n \in (\tfrac{2(s+1)}{ 4(s+1)-d },2)$}

We assume that $\Omega = B_R (0)$. 
We prove the property of finite propagation speed by contradiction. More precisely, we assume that for all
$t \in (0,T]$, the support of $u(\cdot, t)$ is $\bar{\Omega}$, i.\,e., $\operatorname{supp}u(\cdot, t) = \bar{\Omega}$
for every $t \in (0,T]$. We are thus assuming that $u(\cdot, t)>0$ in $\bar{\Omega}$ and for any  $t\in (0,T]$.
The argument will show that there exists $T^{**}>0$ and a function $d=d(T)<R$ such that the entropy concentrated on annuli \eqref{eq:ene_annuli} verifies
$$
A_T(S) := \iint_{\Omega_T(S)}u^{\alpha+2}\d x \d t = 0 \qquad\text{for all }T\in [0,T^{**}] \text{ and for all }  S\geqslant   d(T).
$$
This is clearly in contradiction with the assumed $u >0$ in $\bar{\Omega}$.

The $\alpha$-entropy inequality (\ref{eq:alpha-entropy-u1}) and Sobolev embeddings give that the solution $u$ is {H\"{o}lder continuous in $\bar{\Omega}$ for a.\,e. $t \in (0,T)$. Thus, since $u>0$, by the Weierstrass Extreme value Theorem there exists $ c(t) := \mathop {\min}_{x \in \bar{\Omega}} u^{\frac{\alpha +2}{2}}(t,x) \in L^2(0,T)$ such that
$$
u^{\frac{\alpha +2}{2}}(\cdot, t)\geqslant  c(t) > 0 \qquad \text{ in }  \bar{\Omega} \text{ for  a.\,e. } t \in (0,T].
$$ }
Consequently,
$$
A_T(S) \geqslant  \frac{\pi^{\frac{d}{2}}}{\Gamma( \frac{d}{2} + 1 )}R^{d-1}(R-S)
\int_0^T {{c^2(t)} \,dt} =:  C_0(T) (R - S).
$$
Let $\gamma_T:[0,R]\to [0,+\infty)$ be a non-increasing function such that
\begin{equation}\label{ddr}
\begin{cases}
0 < \gamma_T(0) \leqslant   C_0(T) R, \\
\gamma_T(S) \leqslant   C_0(T) (R-S)  &\text{for all } S\in[0,R],\\
\gamma_T(S) > 0 & \text{for all }  S\in[0,R),\\
\gamma_T(S + \delta)  \leqslant   \kappa  \gamma_T^{\beta + \frac{\beta}{\lambda} }(S)  &\text{for all }  S \geqslant  0, \ \delta > 0,
\end{cases}
\end{equation}
with $0 < \gamma_T(0) < \min\{1, C_0(T) \}R $, where $\lambda>0$ is a free parameter while $\beta$ and $\kappa$ verify
\begin{equation} \label{eq:parameters}
\beta :=  1+\tfrac{n(1-\theta) }{\alpha - n + 2 } > 1, \qquad
\kappa \in \Bigl(0,  \frac{1 - \frac{\gamma_T(0)}{R} }{ \gamma_T^{\beta -1 + \frac{\beta}{\lambda}}(0)} \Bigr).
\end{equation}
We observe that a function satisfying the assumptions (\ref{ddr}) exists, e.\,g. $\gamma_T(S) = \frac{\gamma_T(0)}{R}(R-S)_+^{(\beta + \frac{\beta}{\lambda})^{\frac{S}{\delta}}}$ 
and $R = \kappa^{\frac{1}{\beta + \frac{\beta}{\lambda} -1}} \gamma_T(0) \in (0,1]$.
Thanks to the Stampacchia's lemma (see Lemma~\ref{lem-st-n}) a function satisfying the assumptions \eqref{ddr} also verifies $\gamma_T(S) \equiv 0 $ for all $S \geqslant  R $.
As a result, we have
\begin{equation}\label{as-b}
 A_T(S) \geqslant  \gamma_T(S )   \qquad\text{for all } T >0, \  S \in [0,R],
\end{equation}
and, thus, for any $S\in [0,R]$ and for any $T>0$, calling $\nu := 1-\varpi>0$,
\begin{equation}\label{eq:A_T}
A_T(S) +  A^{1 - \varpi}_T(0) A_T^{\varpi}(S) \leqslant 2 \left ( \tfrac{A_T (0) }{\gamma_T(S)}\right)^{\nu} A_T(S) =: \tilde{C}_T(S) A_T(S).
\end{equation}
Estimate (\ref{main-alp}) becomes
\begin{equation} \label{e-2-004}
\begin{split}
\int_{\Omega(S+ \delta)}{\mathscr{G}_{\alpha}(u) \,dx} & +
C \iint_{\Omega_T(S+ \delta)}{ \bigl|(-\Delta)^{\frac{s+1}{2}}(u^{\frac{\alpha +2}{2}} \psi_{S,\delta } )  \bigr|^2   \,dx dt}  \\
& \leqslant \int_{\Omega(S)}{\mathscr{G}_{\alpha}(u_0) \,dx} +
\tfrac{\tilde{C}_T(S)}{ \delta^{2(s+1)  } } A_T(S)  =: R_T(S,\delta).
\end{split}
\end{equation}

Applying the Gagliardo--Nirenberg type interpolation inequality (see Lemma~\ref{G-N-nn})
in the region $\Omega$ to the function $v(\cdot,t) :=  u^{\frac{\alpha +2}{2}} (\cdot,t)$ with $b = \frac{2(\alpha - n+2)}{\alpha+2}$
 and $\theta  = \tfrac{n d}{n d + 2(s+1) (\alpha -n + 2)}$, after
integrating in time and using (\ref{eq:alpha-entropy-u1}), we get
$$
A_T(0) \leqslant   C\,T^{\mu} \|u_0\|^{(\alpha-  n+2)\beta}_{L^{\alpha-  n+2}(\Omega)} + C\,T  \|u_0\|^{2}_{L^{\alpha-  n+2}(\Omega)}, \ \ \mu := 1-\theta,
$$
whence
\begin{equation}\label{a-in}
A_T(0) \leqslant   C\,T^{\mu} \|u_0\|^{(\alpha-  n+2)\beta}_{L^{\alpha-  n+2}(\Omega)} {\text{ for small enough }} T > 0.
\end{equation}

Next, applying the homogeneous Gagliardo-Nirenberg type interpolation inequality (see Lemma~\ref{G-N-nn})
in the region $\Omega(S+  \delta)$ to the function $v(\cdot,t) := u^{\frac{\alpha +2}{2}}(\cdot,t)\psi _{S+\delta,\delta }(\cdot) $ with 
$b = \frac{2(\alpha - n+2)}{\alpha+2}$ and $\theta  = \tfrac{n d}{n d + 2(s+1) (\alpha -n + 2)}$, we find that
\begin{eqnarray} \label{rrr-ll}
\| u^{\frac{\alpha +2}{2}} \psi _{S+\delta,\delta } \|^{ 2}_{L^{ 2}(\Omega(S+\delta))}
&\leqslant C \|  (-\Delta)^{\frac{s+1}{2}} ( u^{\frac{\alpha +2}{2}} \psi _{S+\delta,\delta } ) \|^{ 2 \theta }_{L^{2}  (\Omega)}  \| u^{\frac{\alpha +2}{2}} \psi _{S+\delta,\delta } \|^{
2(1-\theta) }_{L^{\frac{2(\alpha -n +2)}{\alpha +2} }(\Omega(S+ \delta))}.
\end{eqnarray}
From (\ref{rrr-ll}), using (\ref{rrr-int}), it follows that
\[
\begin{split}
\|   u^{\frac{\alpha +2}{2}} \psi _{S+\delta,\delta } \|^{2}_{L^{2}(\Omega(S+\delta))} &
\leqslant
C \|  (-\Delta)^{\frac{s+1}{2}} ( u^{\frac{\alpha +2}{2}} \psi _{S+\delta,\delta } ) \|^{2\theta }_{L^{2}  (\Omega(S+2\delta))}  \| u^{\frac{\alpha +2}{2}} \psi _{S+\delta,\delta } \|^{ 2(1-\theta) }_{L^{\frac{2(\alpha -n +2)}{\alpha +2} }(\Omega(S+ \delta))} \\
&+  ( \tfrac{C}{  \delta})^{ 2 \theta (s+1 ) } \|    u^{\frac{\alpha +2}{2}} \psi _{S+\delta,\delta }  \|^{ 2 \theta  }_{L^{2}  (\Omega(S+ \delta))}  \| u^{\frac{\alpha +2}{2}} \psi _{S+\delta,\delta } \|^{ 2 (1-\theta) }_{L^{\frac{2(\alpha -n +2)}{\alpha +2} }(\Omega(S+ \delta))},
\end{split}
\]
whence, using  Young's inequality, we arrive at
\begin{equation}\label{int-3}
 \begin{split}
 \|   u^{\frac{\alpha +2}{2}}  \|^{2}_{L^{2}(\Omega(S+ 2\delta))}  &\leqslant 
C \|  (-\Delta)^{\frac{s+1}{2}} ( u^{\frac{\alpha +2}{2}} \psi _{S+\delta,\delta } ) \|^{ 2\theta  }_{L^{2}  (\Omega(S+ 2\delta))}
 \| u   \|^{ (\alpha +2)(1-\theta) }_{L^{ \alpha -n +2}(\Omega(S+ \delta))}   \\
& +\tfrac{C}{  \delta^{\frac{ 2\theta(s+1)}{ 1-\theta  }}}
\| u \|^{ \alpha + 2}_{L^{ \alpha -n +2 } (\Omega(S+ \delta))} .
\end{split}
\end{equation}
Integrating  (\ref{int-3}) with respect to time, using H\"older's inequality and \eqref{e-2-004}, exchanging $S +\delta$
with $S$, we arrive at the following relation:
$$
A_T(S + \delta) \leqslant C\,T^{ 1-\theta  } R^{1+\frac{n(1-\theta) }{\alpha - n + 2  }}_T(S,\delta)
+ \tfrac{C}{  \delta^{\frac{2\theta(s+1)}{ 1-\theta  }}} T  R^{1+\frac{n}{\alpha - n + 2}}_T(S,\delta).
$$
Now, we consider $S \geqslant S_1 \geqslant r_0$ which implies that $u_0\equiv 0$ on $\Omega(S)$, 
 therefore  $\mathscr{G}_{\alpha}(u_0)  = 0$  for all $ x \in \Omega(S)$. Consequently,
$$
R_T(S,\delta) \leqslant   \tfrac{C}{ \delta^{2(s+1)} }  A_T(0).
$$
Therefore, we can find $T^*=T^*(\delta)$ such that, for $T\leqslant   T^*$ (and $S\geqslant  S_1$), 
\begin{equation}\label{eq:fsp_energy1}
A_T(S + \delta) \leqslant   C\,T^{ 1-\theta  } R^{1+\frac{n(1-\theta) }{\alpha - n + 2  }}_T(S,\delta).
\end{equation}
Moreover, due to (\ref{eq:A_T}), we deduce that
\begin{equation}\label{e-12}
A_T(S +\delta) \leqslant  C_4 T^{\mu} \bigl(\tilde{C}_T(S) \delta^{- \tilde{\alpha} } A_T(S) \bigr)^{\beta}
\end{equation}
for all $S \geqslant S_1 \geqslant r_0$, where
$$
\mu =  1-\theta  , \ \tilde{\alpha} = 2(s+1) , \
\beta = 1+\tfrac{n(1-\theta) }{\alpha - n + 2 } > 1.
$$
Hence, calling $ \tilde{A}_T(S ) : = \tilde{C}_T^{ - \lambda}(S) A_T(S) $, by (\ref{e-12})  
we have
\begin{equation}\label{e-12-000}
\tilde{A}_T(S +\delta) \leqslant  C_4 \kappa^{\lambda \nu} T^{\mu} \bigl( \delta^{-\tilde{\alpha}} \tilde{A}_T(S) \bigr)^{\beta}
\end{equation}
for $S \geqslant S_1$. Applying the classical Stampacchia's lemma (see Lemma~\ref{L-1}(i))
to (\ref{e-12-000})
with $\alpha = \tilde{\alpha} \beta$, $\beta  > 1$, and $C = C_4 \kappa^{\lambda \nu}   T^{\mu}$,
we find that
\begin{equation}\label{e-13}
\tilde{A}_T(S) = 0 \ \ \ \forall\, S \geqslant d(T) = S_1 + 2^{\frac{\beta}{\beta -1}} (C_4 \kappa^{\lambda \nu} T^{\mu}  \tilde{A}_T^{\beta -1}(S_1) )^{\frac{1}{\tilde{\alpha}\beta}}.
\end{equation}
As, in view of (\ref{a-in}),
$$
\tilde{A}_T(S_1) = 2^{-\lambda} \bigl ( \tfrac{\gamma_T(S_1)}{A_T(0)} \bigr)^{\lambda \nu} A_T(S_1) \leqslant 
2^{-\lambda} A_T(0) \leqslant C \,  T^{\mu} \|u_0\|_{L^{\alpha -n +2}(\Omega)}^{(\alpha -  n + 2)\beta}  
$$
then  an upper bound for $d(T)$ is
$$
d(T) \leqslant S_1 + C_5 T^{\frac{\mu }{\tilde{\alpha}}  } =
S_1 + C_5  T^{\frac{ 1-\theta  }{  2(s+1)  }  } .
$$
For $T \leqslant   T^{**} := \min \bigl\{T^*, \bigl(\frac{R-S_1 }{C_5}\bigr)^{\frac{2(s+1)}{ 1-\theta } }\bigr\}$ we have that
$$
C_5 \, T^{\frac{  1-\theta  }{  2(s+1)  }  }  < R - S_1.
$$
Thus, for $T\leqslant   T^{**}$ we have $d(T) <R$.
On the other hand, by (\ref{as-b}) we have
$$
\tilde{A}_T(S) \geqslant 2^{-\lambda} A_T^{-\lambda \nu}(0) \gamma_T^{1 + \lambda\nu}(S) \geqslant 
 C \, \gamma_T^{1 + \lambda\nu}(S) > 0 \text{ for all } S \in [0,R). 
$$
So, since $\tilde{A}_T(S) = 0$ for all $S \geqslant  d(T) $,  and  $d(T) <  R $  for all {$T \in [0,T^{**}]$}, we obtain the contradiction to (\ref{as-b}).

Next, we look for an exact estimate for interface speed. For this, it remains to estimate
$\tilde{A}_T(S_1)$  and to choose an optimal value of $S_1$.
Applying the homogeneous Gagliardo-Nirenberg type interpolation inequality (see Lemma~\ref{G-N-nn})
in the region $\Omega(S+\delta)$ to the function $v(\cdot,t) := u^{\frac{\alpha +2}{2}}(\cdot,t)\psi _{S+\delta,\delta }(\cdot) $
with $ b = \frac{2}{\alpha +2} $  and $\theta  = \tfrac{d(\alpha+1)}{d(\alpha+1)+2(s+1)}$, we find that
\[
\begin{split}
\| u^{\frac{\alpha +2}{2}} \psi _{S+\delta,\delta }  \|^2_{L^2(\Omega(S+ \delta))} &\leqslant 
 C\, M^{\frac{4(s+1)}{d(\alpha+1)+ 2(s+1)}}  \|(-\Delta)^{\frac{s+1}{2}} ( u^{\frac{\alpha +2}{2}} \psi _{S+\delta,\delta } ) \|^{\frac{2d(\alpha+1)}{d(\alpha+1)+ 2(s+1)}}_{L^{2}  (\Omega)}  \\
& \leqslant C\, M^{\frac{4(s+1)}{d(\alpha+1) + 2(s+1)}}  \|(-\Delta)^{\frac{s+1}{2}} ( u^{\frac{\alpha +2}{2}} \psi _{S+\delta,\delta } ) \|^{\frac{2d(\alpha+1)}{d(\alpha+1) + 2(s+1)}}_{L^{2}  (\Omega(S+ 2\delta))} \\
&+ \tfrac{C M^{\frac{4(s+1)}{d(\alpha+1) + 2(s+1)}}}{  \delta ^\frac{2d(\alpha+1)(s+1)}{d(\alpha+1) + 2(s+1)} } \| u^{\frac{\alpha +2}{2}} \psi _{S+\delta,\delta }  \|^{\frac{2d(\alpha+1)}{d(\alpha+1) + 2(s+1)}}_{L^{2}  (\Omega(S+ \delta))}  ,
\end{split}
\]
whence
\begin{equation}\label{int-5}
\| u^{\frac{\alpha +2}{2}}  \|^2_{L^2(\Omega(S+2\delta))}   \leqslant  
 C\, M^{\frac{4(s+1)}{d(\alpha+1) + 2(s+1)}}  \|(-\Delta)^{\frac{s+1}{2}} ( u^{\frac{\alpha +2}{2}} \psi _{S+\delta,\delta } ) \|^{\frac{2d(\alpha+1)}{d (\alpha+1) + 2(s+1)}}_{L^{2}  (\Omega(S+ 2\delta))} +
\tfrac{C M^{2}}{ \delta ^{d(\alpha+1)}} .
\end{equation}
Integrating  (\ref{int-5}) with respect to time, using H\"older's inequality and (\ref{e-2-004}),
exchanging $S + \delta$ with $S$, we arrive at the following
relation:
\begin{equation}\label{e-12-2}
A_T(S + \delta) \leqslant  C_6(M) T^{\frac{2(s+1)}{d(\alpha+1) +2(s+1)}} \bigl( \delta^{-\tilde{\alpha}} \tilde{C}_T(S)  A_T(S) \bigr)^{ \frac{d(\alpha+1)}{d(\alpha+1) +2(s+1)}}
\end{equation}
for $S \geqslant r_0  $ and enough small $T >0$. Hence, by (\ref{e-12-2}) for
$ \tilde{A}_T(S ) : = \tilde{C}_T^{ - \lambda }(S) A_T(S) $ we have
\begin{equation}\label{e-12-222}
\tilde{A}_T(S +\delta) \leqslant  C_7(M)  \kappa^{\lambda \nu}  
 T^{\frac{2(s+1)}{d(\alpha+1) +2(s+1)} } \bigl( \delta^{- \tilde{\alpha}} \tilde{A}_T(S) \bigr)^{ \frac{d(\alpha+1)}{d(\alpha+1) +2(s+1)} }
\end{equation}
for $S \geqslant r_0 $. Applying the Stampacchia's lemma \ref{L-1}
to (\ref{e-12-222})
with ${\alpha} =\frac{d\tilde{\alpha}(\alpha+1)}{d(\alpha+1) +2(s+1)}  $, ${\beta} =  \frac{d(\alpha+1)}{d(\alpha+1) +2(s+1)}  < 1$, and
$ 
C =  C_7(M)  \kappa^{\lambda \nu}   T^{\frac{2(s+1)}{d(\alpha+1) +2(s+1)} } 
$, 
 we find that
\begin{equation}\label{int-6}
\tilde{A}_T(S) \leqslant \bigl(C_8(M) T + C\,r_0^\frac{ d \tilde{\alpha}(\alpha+1)}{2(s+1)} \tilde{A}_T(r_0) \bigr) \, S^{- \frac{ d \tilde{\alpha}(\alpha+1)}{2(s+1)}} \leqslant   
 2C_8(M) T \, S^{- \frac{ d \tilde{\alpha}(\alpha+1)}{2(s+1)}} \ \  \forall\,S \geqslant r_0,
\end{equation}
where
$$
C_8(M) = 2^{\frac{d \tilde{\alpha}(\alpha+1) (d(\alpha+1)+2(s+1))}{4(s+1)^2}}C_7^{\frac{ d(\alpha+1)+2(s+1) }{2(s+1)}}(M)  \kappa^{ \frac{ \lambda \nu ( d(\alpha+1)+2(s+1)) }{2(s+1)} }.
$$
Combining (\ref{int-6}) for $S = S_1 - r_0$ with (\ref{e-13}), we get
$$
d(T) \leqslant S_1 + C_9(M)\, T^{\frac{\mu +\beta -1}{ \tilde{\alpha}\beta}}
(S_1 - r_0)^{- \frac{ d \tilde{\alpha}(\alpha+1)}{2(s+1)}}.
$$
Minimizing the right hand side on $S_1$, we find that
$$
d(T) \leqslant r_0 + C(M) T^{\frac{1}{nd + 2(s+1)}}.
$$

\section{Lower-bound on the waiting time} \label{sec:LWT}

We now turn to the proof of the optimal lower bounds on waiting times.

\begin{proof}[Proof of Theorem~\ref{th:wt}]
We assume that $\supp u_0 \subset \Omega \setminus \Omega(r_0)$.
We prove that there exists a time $T_0\in (0,T^{*}]$ such that we have
\begin{align*}
u(\cdot, t) = 0 \quad \text{ in } \Omega(r_0) \quad \text{for}\quad 0 < t < T_0.
\end{align*}
With the same notations as in the proof of Theorem~\ref{th:1}, we recall that we have the refined entropy inequality:
$$
 \int_{\Omega(S+ \delta)}{\mathscr{G}_{\alpha}(u) \,\d x}  +
C \iint_{\Omega_T(S+ \delta)}{ \bigl|(-\Delta)^{\frac{s+1}{2}}(u^{\frac{\alpha +2}{2}} \psi_{S,\delta } )  \bigr|^2   \, \d x \d t}   
 \leqslant     \int_{\Omega( S)}{\mathscr{G}_{\alpha} (u_0) \dd x} + \tfrac{\tilde{C}_T(S)}{ \delta^{2(s+1)  } } A_T(S).
$$
Arguing as in (\ref{e-12-000}),  we arrive at
\begin{align}\label{wtp-inq}
\begin{aligned}
\tilde{A}_T(S+\delta) &\leqslant    C_4 \kappa^{\lambda \nu}  T^{\mu} \left( \delta^{- 2(s+1) } \left( \int_{\Omega(S)} \delta^{2(s+1)} \mathscr{G}_{\alpha}(u_0)  \dd x  + \tilde{A}_T(S) \right) \right)^{\beta} \\
& \leqslant   C_4 \kappa^{\lambda \nu} T^{\mu} \left( \delta^{-2(s+1)} \bigl(  \mathcal{S} \, \delta^{\sigma}  + \tilde{A}_T(S) \bigr) \right)^{\beta} 
\end{aligned}
\end{align}
holds for all $S \in [r_0-\delta, r_0]$, where we introduced the notations
\begin{align*}
&\mu =   1-\theta , \quad \theta  = \frac{nd}{nd + 2(s+1) (\alpha -n +2)}, \quad
\beta  =  1 + \frac{n(1-\theta) }{\alpha -n +2}, \\
& \sigma  = d + 2(s+1) + (\alpha -n +2)\gamma,\quad \mathcal{S} :=  \sup_{\delta>0} \delta^{-\gamma(\alpha -n +2)} \fint_{ B_{r_0}(0) \setminus B_{r_0 -\delta}(0) } \mathscr{G}_{\alpha}(u_0(x))  \dd x.
\end{align*}  

Arguing as in \cite[Theorem 2.1]{GiacomelliGruen} we deduce that 
there exists a time $T_0 > 0$ such that
\begin{align*}
\tilde{A}_{T_0} (r_0) \equiv 0  \Rightarrow  \int_0^{T_0} \int_{ \Omega(r_0)} u^{\alpha +2} (t,x) \dd x \dd t \equiv 0
\end{align*}
provided
$$
\sigma = d + 2(s+1) + (\alpha -n+ 2)\gamma \geqslant   \tfrac{2\beta(s+1)}{\beta-1}, \text{ i.\,e. } \gamma \geqslant   \tfrac{2(s+1)}{n}.
$$
Specifically, we define 
$$
f(\xi) = \tilde{A}_T(S+\delta), \ \ \ f(\eta) = \tilde{A}_T(S), \ \ \ \xi -\eta = \delta.
$$
As a result, inequality (\ref{wtp-inq}) becomes
$$
f(\xi) \leqslant \tfrac{C_4 \kappa^{\lambda \nu} T^{\mu}}{(\xi -\eta)^{2\beta(s+1)}} (f(\eta) + \mathcal{S}(\xi -\eta)^{\sigma})^{\beta}.
$$
To apply the inhomogeneous Stampacchia lemma (Lemma~\ref{lem:stampacchia2}) with 
$$
R = r_0 +\delta, \ \ \ c_0 = C_4 \kappa^{\lambda \nu} T^{\mu}, \ \ \ 
\tilde{\mathcal{S}} = \mathcal{S}, \ \ \ \alpha = 2\beta(s+1), \ \ \ \beta > 1,
$$ 
it suffices to verify that 
$$
(r_0 +\delta)^{\frac{\alpha}{\beta -1}} \geqslant (2^{\frac{\beta(\alpha + \beta -1)}{\beta -1}}C_4 \kappa^{\lambda \nu} T^{\mu} )^{ \frac{1}{\beta -1} } ( \tilde{A}_T(0) + \mathcal{S} (r_0 +\delta)^{\frac{\alpha}{\beta-1}}),
$$
where we take $\sigma \geqslant \frac{\alpha}{\beta-1}$. The above condition reduces to 
$$
(2^{\frac{\beta(\alpha + \beta -1)}{\beta -1}} C_4 \kappa^{\lambda \nu} )^{ \frac{1}{\beta -1} } \mathcal{S}   T^{ \frac{\mu}{\beta -1} }    \leqslant 1, \text{ whence } T  \leqslant (2^{\frac{\beta(\alpha + \beta -1)}{\beta -1}} C_4 \kappa^{\lambda \nu} )^{- \frac{1}{\mu}} \mathcal{S}^{ - \frac{\beta -1}{\mu}} .
$$
Consequently, we obtain the following lower bound for the waiting time: 
$$
T_0 \geqslant \left(2^{\frac{\beta(\alpha + \beta -1)}{\beta -1}}  C_4 \kappa^{\lambda \nu} \right)^{-\frac{1}{\mu}} \mathcal{S}^{ - \frac{\beta -1}{\mu}}
= C\, \mathcal{S}^{-\frac{n}{\alpha -n + 2}},
$$
and we conclude that 
$$
\tilde{A}_T(r_0 +\delta) = 0 \text{ for all } T \in [0,T_0].
$$
Since $\delta > 0$ is arbitrary, by the monotonicity of $\tilde{A}_T(S)$ with respect to $S$, we obtain that $\tilde{A}_T(r_0) = 0$ for all $T \in [0,T_0]$.


\end{proof}

\vspace{5mm}
\section*{Acknowledgments}

Antonio Segatti acknowledges support from PRIN 2022 (Project no. 2022J4FYNJ), funded by MUR, Italy, and the European Union -- Next Generation EU, Mission 4 Component 1 CUP F53D23002760006 and the support of Indam-Gnampa.

Roman Taranets was supported by NRFU project no. 2023.03/0074 ``Infinite-dimensional evolutionary equations with multivalued and stochastic dynamics'' 
and by a grant from the Simons Foundation (Award no. 00017674, Presidential Discretionary Ukraine Support Grants).

\appendix

\section{Auxiliary Lemmas}
\label{app:lemmas}
%


\begin{lemma}[Gagliardo--Nirenberg-type inequality] $($see \cite[Lemma A.1]{DNLST24}$)$\label{G-N-nn}
 If $\Omega  \subset \mathbb{R}^N $ is a bounded
domain with piecewise-smooth boundary, $b \in (0, 2)$  and $s \in (0,1)$, then there
exist  positive constants  $C_1$ and $C_2$ ($C_2 =0$ if $\Omega$ is unbounded or $w(x)$ has a compact support)
depending only on $\Omega ,\ s,\ b,$
and $N$ such that the following inequality is valid for every
$w(x) \in H^{s+1}(\Omega ) \cap L^b (\Omega )$:
$$
\| w  \|_{L^2 (\Omega )}  \leqslant   C_1  \|
(-\Delta)^{\frac{s+1}{2}} w  \|_{L^2 (\Omega )}^\theta  \| w  \|_{L^b
(\Omega )}^{1 - \theta } + C_2 \| w  \|_{L^b
(\Omega )}, \quad  \theta  = \frac{{\frac{1} {b}   - \frac{1}
{2}}} {{\frac{1} {b} + \frac{s+1} {N} - \frac{1} {2}}} \in  [0,1 ) .
$$
\end{lemma}

\begin{lemma}[Classical Stampacchia's lemma] $($see \cite[Lemme 4.1,p.14]{S1}$)$\label{L-1}
Let $f(x)  $ be non-negative, non-increasing in $[x_0,+\infty)$ function. Assume that $f$ satisfies
\begin{equation}\label{s-1}
 f(y) \leqslant \frac{C}{(y -x)^{\alpha}} f^{\beta}(x)  \text{ for } y > x \geqslant x_0,
\end{equation}
where $C,\,\alpha,\, \beta$ are some positive constants. Then

(i) if  $\beta > 1$ we have
$$
f(y) = 0  \text{ for all }  y \geqslant x_0 + d,
$$
where $d^{\alpha} = C f^{\beta -1}(x_0) 2^{\frac{\alpha\beta}{\beta -1}}$;

(ii) if  $\beta = 1$ we get
$$
f(y) \leqslant e^{1- \zeta(y -x_0)}f(x_0) \text{ for all }  y \geqslant x_0,
$$
where $ \zeta = (e\,C)^{- \frac{1}{\alpha}}$;

(iii) if  $\beta < 1$ we obtain
$$
f(y) \leqslant  2^{\frac{\mu}{1-\beta}} \bigl[C^{\frac{1}{1-\beta}} + (2\, x_0)^{\mu} f(x_0) \bigr] y^{-\mu} \text{ for all }  y \geqslant x_0 > 0,
$$
where $ \mu =  \frac{\alpha }{1- \beta}$.
\end{lemma}

\begin{lemma}[Inhomogeneous Stampacchia's lemma]$($see \cite[Lemma 3.1]{DalPassoGiacomelliGruen}, \cite[Lemma 2.4]{GiacomelliGruen}$)$\label{lem:stampacchia2}
Let $f:[0,R] \to \R$ be a non-negative non-increasing function such that
	\begin{equation}\label{stamglem:h21}
	f(\xi) \leqslant  \frac{c_0}{(\xi-\eta)^\alpha}(f(\eta)+\widetilde{S}\cdot (R-\eta)^{\frac{\alpha}{\beta-1}})^\beta,
	\end{equation}
	where $0 \leqslant  \eta < \xi \leqslant   R$, $\widetilde{S} \geqslant  0$, and $c_0,\alpha,\beta >0$, and $\beta >1$. In addition, let us assume that
	\begin{equation}\label{stamglem:h22} R^{\frac{\alpha}{\beta -1}} \geqslant  \left( 2^{\frac{\beta(\alpha + \beta -1)}{\beta -1 }} c_0\right)^{\frac{1}{\beta-1}}  (f(0) + \widetilde{S} \cdot R^{\frac{\alpha}{\beta -1}}).
	\end{equation}
	Then $$f(R) = 0.$$
\end{lemma}

\begin{lemma}$($see \cite[Lemma A.7]{DNLST24}$)$\label{lem-fr}
Assume that $\Omega \subset \mathbb{R}^d$ is a bounded domain  and $\Omega(S) := \Omega \setminus \{ x \in \bar{\Omega}:  |x | \leqslant S \}$.
Let $\psi(x) \in H_N^{2\mu}( \Omega )$ be such that $\text{supp}\,\psi(x) \subseteq \Omega(S) $.
Then there exists a constant $C > 0$ depending on $\mu$ such that  the following estimate holds
\begin{equation}\label{fr-1}
\| (-\Delta)^{\mu} \psi  \|_{L^2(\Omega)} \leqslant \| (-\Delta)^{\mu} \psi  \|_{L^2(\Omega(S+\delta))} +  \tfrac{C}{\delta^{2\mu }}\|\psi  \|_{L^2(\Omega(S))} ,
\end{equation}
where $\mu \in (0,1)$.
\end{lemma}
 
\begin{lemma}[Stampacchia-type lemma]$($see \cite[Lemma A.3]{DNLST24}$)$\label{lem-st-n}
Let  $f : [0,+\infty) \to [0, +\infty)$ be a non-negative non-increasing function
such that
\begin{equation}\label{nnn-1}
f(s +\delta) \leqslant   \epsilon f^{\nu}(s) \text{ for all } s \in \mathbb{R}^+, \ \delta > 0,
\end{equation}
for  $\epsilon \in (0,f^{1-\nu}(0))$ and $\nu > 1$. Then
$$
f(s) \equiv 0 \text{ for all } s \geqslant  d :=  \frac{f(0)}{1 - \epsilon f^{\nu -1}(0)}.
$$
\end{lemma}

\section{Proof of Lemma \ref{lem-GG}}

%
%
%
%

We start with the proof of \eqref{gg-0-01}.
We recall that  for $\mu \in (0,1)$ we have
$$
(-\Delta)^{\mu} \psi(x) = \tfrac{1}{\Gamma(-\mu)} \int_0^{+\infty} { (e^{t\Delta}\psi(x) - \psi(x) ) \tfrac{dt}{t^{1+\mu}}},
$$
where
$$
e^{t\Delta}\psi(x) = \int_{\Omega} { K(x,y,t) \psi(y)\,dy}, \ \ \int_{\Omega} { K(x,y,t)\,dy} = 1.
$$
We also recall \eqref{control_kernel} which we express as 
$K(x,y,t) \backsimeq C_K t^{-\frac{d}{2}} e^{-\frac{|x-y|^2}{4t}}$ for a suitable constant $C_K$.  
For every $x \in \Omega$, we have
$$
e^{t\Delta}\phi(u(x)) - \phi(u(x)) = \int_{\Omega} { K(x,y,t) (\phi(u(y)) - \phi(u(x)) )\,dy}.
$$
By the Taylor expansion with integral remainder of $\phi(u)$, we can write
$$
\phi(u(y)) - \phi(u(x)) = \phi'(u(x)) \bigl(u(y) - u(x) \bigr) +
 \int_{u(x)}^{u(y)} {  \phi''(z)   (u(y) - z) \,dz}.
$$
Substituting it in the above identity, we find
\[
\begin{split}
 e^{t\Delta}\phi(u(x)) - \phi(u(x)) &= \phi'(u(x)) \int_{\Omega} { K(x,y,t) \bigl(u(y) - u(x) \bigr) \,dy} \\
&\,\,\, +
  \int_{\Omega} { K(x,y,t)  \int_{u(x)}^{u(y)} {  \phi''(z)   (u(y) - z) \,dz}  \,dy}.
  \end{split}
\]
Multiplying both sides of this equation by $ \frac{1}{\Gamma(-\mu)} t^{-1-\mu}$  and integrating in $t$  over $(0,+\infty)$,
we arrive at (\ref{gg-0-01}) with
$$
\mathcal{I}_\mu[u](x) = -  \tfrac{1}{ \Gamma(-\mu)} \int_0^{+\infty} { \Bigl(
\int_{\Omega} { K(x,y,t) \int_{u(x)}^{u(y)} {  \phi''(z)   (u(y) - z) \,dz} dy}   \Bigr) \tfrac{dt}{t^{1+\mu}}} . 
$$
Finally, we prove \eqref{gg-0-01-000}. 
 If $\mu \in (1,2)$ then
\[
\begin{split}
(-\Delta)^{\mu} \phi(u(x)) &=(-\Delta) \circ (-\Delta)^{\mu -1 }\phi(u(x))    \\
&= (-\Delta) [\phi'(u(x)) (-\Delta)^{\mu -1 }u(x)  - \mathcal{I}_{\mu-1} [u](x)  ] \\ 
 &=
\phi'(u(x)) (-\Delta)^{\mu} u(x) + (-\Delta) \phi'(u(x)) (-\Delta)^{\mu -1 }u(x) -
2 \nabla \phi'(u(x)) \nabla (-\Delta)^{\mu -1 }u(x)\\
&\quad - (-\Delta)\mathcal{I}_{\mu-1} [u](x).
\end{split}
\]
Now, we compute $ (-\Delta)\mathcal{I}_{\mu-1} [u](x) $:
\[
\begin{split}
(-\Delta)\mathcal{I}_{\mu-1} [u](x) & =  - \tfrac{1}{ \Gamma(1-\mu)}(-\Delta)_x \int_0^{+\infty} { \Bigl(
\int_{\Omega} { K(x,y,t) \int_{u(x)}^{u(y)} {  \phi''(z)   (u(y) - z) \,dz} dy}   \Bigr) \tfrac{dt}{t^{\mu}}} \\
&=  
- \tfrac{1}{ \Gamma(1-\mu)} \int_0^{+\infty} { \Bigl(
\int_{\Omega} { (-\Delta)_x (K(x,y,t)) \int_{u(x)}^{u(y)} {  \phi''(z)   (u(y) - z) \,dz} dy}   \Bigr) \tfrac{dt}{t^{\mu}}} \\
&\quad -  \tfrac{1}{ \Gamma(1-\mu)}  \int_0^{+\infty} { \Bigl(
\int_{\Omega} {  K(x,y,t) (-\Delta)_x \int_{u(x)}^{u(y)} {  \phi''(z)   (u(y) - z) \,dz} dy}   \Bigr) \tfrac{dt}{t^{\mu}}} \\
&\quad + \tfrac{2}{ \Gamma(1-\mu)} \int_0^{+\infty} { \Bigl(
\int_{\Omega} { \nabla_x (K(x,y,t)) \nabla_x \int_{u(x)}^{u(y)} {  \phi''(z)   (u(y) - z) \,dz} dy}\Bigr)
 \tfrac{dt}{t^{\mu}}}\\
 &  =: I_1 + I_2 + I_3.
\end{split}
\]
Here,
\[
\begin{split}
I_1 &= 
  \tfrac{1}{ \Gamma(1-\mu)} \int_0^{+\infty} { \left(
\int_{\Omega} {  K_t(x,y,t) \int_{u(x)}^{u(y)} {  \phi''(z)   (u(y) - z) \,dz} dy}   \right) \tfrac{dt}{t^{\mu}}}  \\
&= \tfrac{\mu}{ \Gamma(1-\mu)} \int_0^{+\infty} { \left(
\int_{\Omega} {  K (x,y,t) \int_{u(x)}^{u(y)} {  \phi''(z)   (u(y) - z) \,dz} dy}   \right) \tfrac{dt}{t^{\mu +1}}}  =
\tfrac{-\mu \Gamma(-\mu)}{ \Gamma(1-\mu)} \mathcal{I}_\mu[u](x) = \mathcal{I}_\mu[u](x),
\end{split}
\]
\[
\begin{split}
  I_2 & =  (-\Delta) \phi'(u(x))  \tfrac{1}{ \Gamma(1-\mu)}  \int_0^{+\infty} { \Bigl(
\int_{\Omega} { K(x,y,t) (u(y) - u(x)) dy}   \Bigr) \tfrac{dt}{t^{\mu}}} \\
&\quad + \phi''(u(x)) |\nabla u(x)|^2  \tfrac{1}{ \Gamma(1-\mu)}  \int_0^{+\infty} { \Bigl(
\int_{\Omega} {  K(x,y,t) dy}   \Bigr) \tfrac{dt}{t^{\mu}}}  \\
& = (-\Delta) \phi'(u(x)) (-\Delta)^{\mu -1} u(x) + \phi''(u(x)) |\nabla u(x)|^2  \tfrac{1}{ \Gamma(1-\mu)}  \int_0^{+\infty} { \Bigl(
\int_{\Omega} {  K(x,y,t) dy}   \Bigr) \tfrac{dt}{t^{\mu}}},
\end{split}
\]
\[
\begin{split}
  I_3 & =  - \nabla \phi'(u(x))  \tfrac{2}{ \Gamma(1-\mu)}  \int_0^{+\infty} { \Bigl(
\int_{\Omega} { \nabla_x( K(x,y,t)) (u(y) - u(x)) dy}   \Bigr) \tfrac{dt}{t^{\mu}}}\\
& = 
- 2 \nabla \phi'(u(x)) \nabla (-\Delta)^{\mu -1} u(x)  - \phi''(u(x)) |\nabla u(x)|^2  \tfrac{2}{ \Gamma(1-\mu)}  \int_0^{+\infty} { \Bigl(
\int_{\Omega} {  K(x,y,t) dy}   \Bigr) \tfrac{dt}{t^{\mu}}} .
\end{split}
\]
Using these equalities, we have
\[
\begin{split}
(-\Delta)\mathcal{I}_{\mu-1} [u](x)& = \mathcal{I}_\mu[u](x)  + (-\Delta) \phi'(u(x)) (-\Delta)^{\mu -1} u(x)  -
2 \nabla \phi'(u(x)) \nabla (-\Delta)^{\mu -1} u(x)     \\
&\quad -\phi''(u(x)) |\nabla u(x)|^2  \tfrac{1}{ \Gamma(1-\mu)}  \int_0^{+\infty} { \Bigl( \int_{\Omega} {  K(x,y,t) dy}   \Bigr) \tfrac{dt}{t^{\mu}}}.
\end{split}
\]
As a result, we arrive at
$$
(-\Delta)^{\mu} \phi(u(x)) =  \phi'(u(x)) (-\Delta)^{\mu} u(x)   -  \mathcal{J}_{\mu } [u](x),
$$
where
$$
\mathcal{J}_{\mu } [u](x) := \mathcal{I}_{\mu } [u](x) - \phi''(u(x)) |\nabla u(x)|^2  \tfrac{1}{ \Gamma(1-\mu)}  \int_0^{+\infty} { \Bigl( \int_{\Omega} {  K(x,y,t) dy}   \Bigr) \tfrac{dt}{t^{\mu}}}.
$$

\vspace{5mm}

\bibliographystyle{abbrv}
\bibliography{FTFE-refs_entropy.bib}
\vfill

%
%
%
%
%
%
%

\end{document}